\documentclass[10pt,reqno,final]{amsart}
\usepackage{epsfig,amssymb,amsmath,version}
\usepackage{amssymb,version,graphicx,fancybox,mathrsfs,multirow}
\usepackage{url,hyperref}
\usepackage[notcite,notref]{showkeys}

\usepackage{subfigure}
\usepackage{color,xcolor}
\usepackage{cases}
\usepackage{mathtools}

\catcode`\@=11 \theoremstyle{plain}
\@addtoreset{equation}{section}   
\renewcommand{\theequation}{\arabic{section}.\arabic{equation}}
\@addtoreset{figure}{section}
\renewcommand\thefigure{\thesection.\@arabic\c@figure}
\renewcommand{\thefigure}{\arabic{section}.\arabic{figure}}
\newtheorem{thm}{\bf Theorem}

\newenvironment{theorem}{\begin{thm}} {\end{thm}}
\newtheorem{cor}{\bf Corollary}
\newtheorem{prop}{Proposition}[section]

\newtheorem{lmm}{\bf Lemma}

\newenvironment{lemma}{\begin{lmm}}{\end{lmm}}
\theoremstyle{remark}
\newtheorem{rem}{\bf Remark}[section]
\theoremstyle{definition}
\newtheorem{defn}{\bf Definition}[section]

 \numberwithin{table}{section}

\def \ri {{\rm i}}
\def \IM {\Im}
\def \RE {\Re}

\renewcommand   \varphi  \theta
\renewcommand   \vartheta \phi

\renewcommand \wedge \times

\begin{document}
\bibliographystyle{plain}

\baselineskip 13pt

\title[Asymptotic study] {Uniform bounds and  asymptotics of Generalized Gegenbauer functions of fractional degree}
\author[
	W. Liu \, $\&$\,    L. Wang
	]{\;\; Wenjie Liu${}^{1,2}$   \;\;  and\;\; Li-Lian Wang${}^{2}$
		}
	
	\thanks{${}^{1}$Department of Mathematics, Harbin  Institute of Technology, 150001, China.
The research of this author is partially  supported by the China Postdoctoral Science Foundation Funded Project (No. 2017M620113), the National Natural Science Foundation of China (No. 11801120) and the Fundamental Research Funds for the Central Universities (Grant No.HIT.NSRIF.2019058). \\
		\indent ${}^{2}$Division of Mathematical Sciences, School of Physical
		and Mathematical Sciences, Nanyang Technological University,
		637371, Singapore. The research  is partially supported by Singapore MOE AcRF Tier 1 Grant (RG 27/15), and
		 Singapore  MOE AcRF Tier 2 Grant (MOE2017-T2-2-144).
	}
	
\begin{abstract}
The generalised  Gegenbauer functions of fractional degree (GGF-Fs), denoted by  ${}^{r\!}G^{(\lambda)}_\nu(x)$ (right GGF-Fs) and  ${}^{l}G^{(\lambda)}_\nu(x)$ (left GGF-Fs)  with $x\in (-1,1),$  $\lambda>-1/2$  and real $\nu\ge 0,$  are  special
functions  (usually non-polynomials), which are defined upon  the hypergeometric representation of the  classical Gegenbauer polynomial by allowing integer degree to be real fractional degree.  Remarkably,  the GGF-Fs become  indispensable for optimal error estimates of polynomial approximation to singular functions, and have intimate relations with several  families of nonstandard basis functions recently introduced for solving fractional differential equations.
		However, some  properties of GGF-Fs, which are important pieces for the analysis and  applications, are unknown or under explored. The purposes of this paper are twofold.
The first is to show that for $\lambda,\nu>0$ and $x=\cos\theta$ with $\theta\in (0,\pi),$
	\begin{equation*}\label{IntRep-0N}
	(\sin \varphi)^{\lambda}\,{}^{r\!}G_\nu^{(\lambda)}(\cos \varphi)=
	\frac{2^\lambda\Gamma(\lambda+1/2)}{\sqrt{\pi} {(\nu+\lambda)^{\lambda}}} \, {\cos ((\nu+\lambda)\varphi- \lambda\pi/2)}
	+{\mathcal  R}_\nu^{(\lambda)} (\varphi),
	\end{equation*}
 and derive the  precise expression of  the ``residual" term ${\mathcal  R}_\nu^{(\lambda)} (\varphi).$
With this at our disposal, we   obtain  the  bounds of GGF-Fs
 uniform in $\nu.$  Under an appropriate weight function,  the bounds are uniform for  $\theta\in [0,\pi]$ as well.   
Moreover, we can study  the  asymptotics of GGF-Fs with large fractional degree $\nu.$
The second is to  present miscellaneous properties of GGF-Fs for better understanding of this family of useful special functions.
\end{abstract}
\maketitle

\vspace*{-10pt}

\section{Introduction}\label{sect:introduction}

Undoubtedly,  polynomial approximation theory occupies a central place in algorithm development and numerical analysis of  perhaps most of computational  methods. Indeed, one finds  numerous approximation results in various senses documented in a large volume of literature, which particularly include orthogonal polynomial approximation results  related to spectral methods and $hp$-version finite element methods (see, e.g.,   \cite{Babuska1987MMAN,Schwab1998Book,Trefethen2013Book,Shen2011Book} and the  references therein).   Typically, such results are  established in   Jacobi-weighted Sobolev spaces with  integral-order regularity exponentials (see, e.g., \cite{Shen2011Book}),  or   weighted Besov spaces with fractional regularity exponentials  using the notion of   space interpolation
(see, e.g., \cite{Babuskaa2000NM,Babuska2001SINUM,Babuska2002MMMAS}).  In a very recent work \cite{Liu2017arXiv}, we introduced a new framework of  fractional  Sobolev-type spaces involving  Riemann-Liouville (RL) fractional integrals and derivatives in the study of polynomial approximation to singular functions.  Such spaces are naturally arisen
from exact representations of  orthogonal polynomial  expansion coefficients, and could  best characterize the fractional differentiability/regularity, leading to  optimal error estimates.
A very important  piece of the puzzle therein is the so-called GGF-Fs that generalize the classical Gegenbauer polynomials of integer degree to  functions of fractional degree.
It is noteworthy that the GGF-Fs can be generalized by different  means, e.g., the Rodrigues' formula and hypergeometric  function representation.
For instance,  the right GGF-F: ${}^{r\!}G^{(\lambda)}_\nu(x)$  can be viewed as special $g$-Jacobi functions (see Mirevski et al \cite{Mirevski2007AMC}), defined by replacing the integer-order derivative in
 the  Rodrigues' formula of the Jacobi polynomials by the RL fractional derivative.  However,  both the  definition and  derivation of some properties in \cite{Mirevski2007AMC} have flaws (see Remark \ref{correctA}). 
On the other hand, the Handbook  \cite[(15.9.15)]{Olver2010Book}  listed ${}^{r\!}G^{(\lambda)}_\nu(x)$ but  without
presented  any of their properties.
Interestingly,  as pointed out in  \cite{Liu2017arXiv}, the GGF-Fs have a direct bearing on
 Jacobi polyfractonomial (cf. \cite{Zayernouri2013JCP}) and generalised Jacobi functions
  (cf. \cite{Guo2009ANM,Chen2016MC}) recently introduced in developing efficient spectral methods for fractional differential equations.
   It is also noteworthy that the seminal work of Gui and Babu\v{s}ka  \cite{Gui1986NM} on
     $hp$-estimates of Legendre approximation of singular functions essentially relied  on some non-classical Jacobi polynomials with the parameter $\alpha$ or $\beta<-1,$ which turned out  closely related to  GGF-Fs.
In a nutshell, the GGF-Fs (and more generally the generalised Jacobi functions of fractional degree) can be of great value for numerical analysis and computational algorithms, but many of their properties   are still  under explored.

It is known that the study of asymptotics has been a longstanding  subject of  special functions and their far reaching applications (see, e.g., \cite{Olver1974Book,Temme1996Book,Olver2010Book}).  Most of the asymptotic results of classical orthogonal polynomials   can be found in the books   \cite{Szego1975Book,Olver2010Book}, and  are reported in the review papers \cite{Lubinsky2000,Wong2000,Wong2017} in more general senses.    We highlight that the asymptotic formulas of the hypergeometric function: ${}_2F_1(a-\mu,b+\mu; c;  (1-z)/2)$ in terms of  Bessel functions  for large $\mu,$ were derived in  Jones \cite{Jones2001MMAS} following the idea of Olver \cite{Olver1974Book} using differential equations, where the representations with fewer restrictions on the parameters different from those in Watson \cite{Watson1918}  could be obtained.  Farid Khwaja and  Olde Daalhuis  \cite{Khwaja2014AA}  discussed asymptotics of  ${}_2F_1(a-e_1\mu,b+e_2\mu; c+e_3\mu;  (1-z)/2)$ with $e_j=0,\pm 1, j=1,2,3$  in terms of Bessel functions by using the  contour integrals.

 One of the main objectives of this paper is to derive the uniform bounds for the GGF-Fs, which  are  valid  for real degree $\nu>0$  with fixed $\lambda,$ and
also for all $\theta\in [0,\pi]$ but with a suitable weight function to absorb the singularities at the endpoints.
As such, we  can obtain the asymptotic formulas for large degree $\nu,$ and some other useful estimates of the GGF-Fs.
Our delicate analysis is based on an integral representation from a very useful fractional integral formula  in  \cite{Liu2017arXiv} (see \eqref{FCI++} and Lemma \ref{importformLa}).
 In fact, the
Watson's Lemma  and  asymptotic analysis for  Legendre polynomials  (cf. \cite{Olver1974Book}) indeed cast light on our study.  It is important to point out the GGF-Fs are defined as hypergeometric functions with special parameters   (see Definition \ref{defnGGG}), so some asymptotic results follow from  \cite{Jones2001MMAS,Khwaja2014AA} for large parameters in terms of Bessel functions.  However,  we intend to derive the results uniform for the degree and the variable, and the estimates for large parameters are directly consequences.
In other words, our study can lead to different and more explicitly informative estimates.
As such, the results herein can offer useful tools for  analysis of polynomial approximation and applications of this family of special functions. 
A second purpose of this paper is to present various  properties of GGF-Fs. These particularly
include   singular  behaviors of GGF-Fs in the vicinity of the endpoints, and useful fractional calculus formulas.

 The paper is organized as follows.   In  Section \ref{mainsect2}, we first introduce the definition of GGF-Fs, and then present the main results. We then provide  their  proofs in Section \ref{sect:proofs}.
In the last section, we present assorted properties of GGF-Fs for better understanding of this family of special functions.

\section{Main result and its proof}\label{mainsect2}

\subsection{Generalised Gegenbauer functions of fractional degree}  Different from  Mirevski et al \cite{Mirevski2007AMC}, we follow
  \cite{Liu2017arXiv} to define two types of  GGF-Fs by  the hypergeometric function.
  \begin{defn}\label{defnGGG}{\em
   For real $\lambda>-1/2,$ we define the right GGF-F on $(-1,1)$ of real degree $\nu\ge 0$ as
\begin{equation}\label{rgjfdef}
{}^{r\!}G_\nu^{(\lambda)}(x)=\, {}_2F_1\Big(\!\!-\nu, \nu+2\lambda;\lambda+\frac 1 2;\frac{1-x} 2\Big)=1+\sum_{j=1}^\infty \frac{(-\nu)_j(\nu+2\lambda)_j}{
	j!\; (\lambda+1/2)_j }\Big(\frac{1-x}{2}\Big)^j,
\end{equation}
and the left  GGF-F  of real degree $\nu\ge 0$  as
\begin{equation}\label{lgjfdef}
\begin{split}
{}^{l}G_\nu^{(\lambda)}(x)= (-1)^{[\nu]}\, {}_2F_1\Big(\!\!-\nu, \nu+2\lambda;\lambda+\frac 1 2;\frac{1+x} 2\Big)= (-1)^{[\nu]}\,  {}^{r\!}G_\nu^{(\lambda)}(-x),
\end{split}
\end{equation}
where $[\nu]$ is the largest integer $\le \nu,$ and the Pochhammer symbol: $(a)_j=a(a+1)\cdots (a+j-1).$}
\end{defn}
In the above,  the  hypergeometric function is a power series given by
\begin{equation}\label{hyperboscs}
{}_2F_1(a,b;c; z)=1+\sum_{j=1}^\infty \frac{(a)_j(b)_j}{(c)_j}\frac{z^j}{j!},
\end{equation}
where $a,b,c$ are real,  and   $-c\not \in {\mathbb N}:=\{1,2,\cdots \}$  (see, e.g.,  \cite{Andrews1999Book}).  

Note that if  $\nu=n\in {\mathbb N}_0:=\{0\}\cup {\mathbb N}$,  we have
\begin{equation}\label{obsvers0}
{}^{r\!}G_n^{(\lambda)}(x)=  {}^lG_n^{(\lambda)}(x)=G_n^{(\lambda)}(x)=\frac{P_n^{(\lambda-1/2,\lambda-1/2)}(x)} {P_n^{(\lambda-1/2,\lambda-1/2)}(1)},\;\;\; \lambda>-\frac 1 2,
\end{equation}
where  $P_n^{(\alpha,\beta)}(x)$ is  the classical  Jacobi polynomial as defined in  Szeg\"o \cite{Szego1975Book}.   For  $\lambda=1/2,$  the right GGF-F  turns to be  the Legendre function (cf. \cite{Temme1996Book}):   ${}^{r}G_\nu^{(1/2)}(x)=P_\nu(x).$
For $\lambda=0,$ we  have
\begin{equation}\label{Chebnudefn}
{}^{r\!}G_\nu^{(0)}(x)={}^{r\!}G_\nu^{(0)}(\cos \varphi)=\cos (\nu \varphi)= \cos (\nu\,  {\rm arccos} x):=T_\nu(x),
\end{equation}
thanks to the property (cf. \cite[(15.1.17)]{Abramowitz1972Book}):
\begin{equation}\label{newformulas}
{}_2F_1(-a, a,1/2\,;\, \sin^2t)=\cos(2 a t),\quad  a, t\in {\mathbb R}:=(-\infty,\infty).
\end{equation}
\begin{rem}\label{specialGGFs}{\em
 The GGF-Fs ${}^{r\!}G_{n-\lambda+1/2}^{(\lambda)}(x)$ and  ${}^{l}G_{n-\lambda+1/2}^{(\lambda)}(x)$ with integer $n$ up to some constant multiple,
    coincide with some nonstandard singular basis functions introduced in \cite{Zayernouri2013JCP,Chen2016MC} for accurate solution of fractional differential equations.}
 \end{rem}

Inherited from the Bateman's fractional integral formula for  hypergeometric functions  (cf. \cite[P. 313]{Andrews1999Book}),  we can derive the following very useful formula (cf. \cite[Thm. 3.1]{Liu2017arXiv}):  for $\lambda>-1/2,$ and real $   \nu\ge s\ge 0,$
\begin{equation}\label{FCI++}
	\begin{split}
	{}_{x} I_{1}^{s}\big\{(1-x^2)^{\lambda-1/2}
	\,{}^{r\!}G_{\nu}^{(\lambda)}(x)\big\}
	&= \frac 1 {\Gamma(s)}\int_{x}^1 \frac{(1-y^2)^{\lambda-1/2}
	\,{}^{r\!}G_{\nu}^{(\lambda)}(y)}{(y-x)^{1-s}} dy
	\\&=\frac{\Gamma(\lambda+1/2)}{2^s\Gamma(\lambda+s+1/2)}\, (1-x^2)^{\lambda+s-1/2}\, {}^{r\!}G_{\nu-s}^{(\lambda+s)}(x),\;\;
	\end{split}
	\end{equation}
where ${}_{x} I_{1}^{s}$ is the right-sided  RL   fractional integral operator defined  by
\begin{equation}\label{IsIntegral}
{}_{x}I_{1}^s\, u (x)=\frac 1 {\Gamma(s)}\int_{x}^1 \frac{u(y)}{(y-x)^{1-s}} dy.
 \end{equation}
 Note that a similar formula is available for the left GGF-F  ${}^{l}G_{\nu}^{(\lambda)}(x)$  but associated with the left-sided  RL fractional integral.

Thanks to \eqref{FCI++}, we can  derive  the following formula    crucial for the analysis.
 \begin{lemma}\label{importformLa} For real $\nu, \lambda\ge 0,$ we have
	\begin{equation}\label{IntRep}
	(\sin \varphi )^{2\lambda-1}\,{}^{r\!}G_\nu^{(\lambda)}(\cos \varphi)=
	\frac{2^\lambda\,\Gamma(\lambda+1/2)}{\sqrt{\pi}\, \Gamma(\lambda)}\int^\varphi_0\frac{\cos((\nu+\lambda)\vartheta)}{(\cos \vartheta-\cos\varphi)^{1-\lambda}}\,d\vartheta,
	\end{equation}
	for any $\varphi\in (0,\pi).$
 \end{lemma}
 \begin{proof}
 From \eqref{Chebnudefn} and \eqref{FCI++} with $\lambda=0$,   we obtain immediately  that for $\nu\ge s\ge 0,$
	\begin{equation}\label{intGLfs}
	(1-x^2)^{s- 1/2} \,  {}^{r\!}G_{\nu-s}^{(s)}(x)=\frac{2^s\,\Gamma(s+1/2)} {\sqrt \pi\, \Gamma(s)}\int_x^1 \frac{1}{(y-x)^{1-s}}  \frac{T_\nu (y)} {\sqrt{1-y^2}}\, dy.
	\end{equation}
Substituting $s$ and $\nu$ in the above identity by $\lambda$ and $\nu+\lambda,$ respectively, and using a change of variables: $x=\cos \varphi$  and $y=\cos\vartheta,$ we derive  \eqref{IntRep} from \eqref{intGLfs} straightforwardly.
\end{proof}
\begin{rem}\label{marknote}{\em
If  $\lambda=1/2$ and $\nu=n, $ the identity \eqref{IntRep} leads to the first Dirichlet-Mehler formula for the Legendre polynomial  {\rm(cf. \cite[(6.51)]{Temme1996Book}):}
\begin{equation}\label{formulaA}
P_n(\cos\theta) =\frac{\sqrt 2}\pi \int_0^\theta \frac{\cos(n+ 1/2) \phi }{\sqrt{\cos \phi-\cos\theta}}\, d\phi,\;\;\;\;\;  \theta\in (0,\pi),\;\;  n\in {\mathbb N}_0.
\end{equation}
One approach to obtain the asymptotic formula for Legendre polynomial with  $n\to \infty$  is based on this formula, and the Watson's lemma {\rm(cf. \cite[P. 113]{Olver1974Book}).}
This useful argument  indeed  sheds light on the study of GGF-Fs herein.   However,  we aim to study  the behaviour of GGF-Fs uniform for all $\nu,$ so  the route appears very different, delicate and more involved.
 }
\end{rem}

\subsection{Main results}  We first state the results, whose proofs are given
in Section \ref{sect:proofs}. 	 Here, we just consider the right GGF-Fs, but thanks to  \eqref{lgjfdef}, similar results can be obtained for  the left counterparts.
\begin{thm}
	\label{AsympGeg-A}
	For   $\lambda >0$ and  $\varphi\in (0,\pi),$  we have 
	\begin{equation}\label{IntRep-0}
	\begin{split}
	(\sin \varphi)^{\lambda}\,{}^{r\!}G_\nu^{(\lambda)}(\cos \varphi)=
	\frac{2^\lambda\Gamma(\lambda+1/2)}{\sqrt{\pi} {(\nu+\lambda)^{\lambda}}} \, {\cos ((\nu+\lambda)\varphi- \lambda\pi/2)}
	+{\mathcal  R}_\nu^{(\lambda)} (\varphi),
	\end{split}\end{equation}
	where the ``residual" term  ${\mathcal  R}_\nu^{(\lambda)} (\varphi)$ with a representation given by  \eqref{ResidualRnu}, and   there holds
\begin{equation}\label{uniformBndR}
 |{\mathcal  R}_\nu^{(\lambda)} (\varphi)|\le {\mathcal  S}_\nu^{(\lambda)} (\varphi), \quad \forall\, \theta\in (0,\pi).
\end{equation}
  Here, the bound  ${\mathcal S}_\nu^{(\lambda)}(\theta) $ is  given by
	\begin{itemize}\item[(i)]
for $0<\lambda\le 2, \nu+\lambda>1$ and $\nu>0,$	
		\begin{equation}\begin{split}\label{BoundIntFunft-A}	
		 {\mathcal S}_\nu^{(\lambda)} (\varphi)= \frac{\lambda|\lambda-1|2^\lambda\Gamma(\lambda+1/2)}{\sqrt{\pi}(\nu+\lambda-1)^{\lambda+1}} \Big\{|\cot \varphi|+\frac 2 3 \frac{\lambda+1} {\nu+\lambda-1}\Big\};
		\end{split}\end{equation}
		\item[(ii)]  for $\lambda> 2, \nu>\lambda-3$ and $\nu>0,$
		\begin{equation}\label{BoundIntFunft-B}
		\begin{split}
		{\mathcal S}_\nu^{(\lambda)} (\varphi)
		&=
		 \frac{\lambda(\lambda-1)2^{3\lambda/2}\Gamma(\lambda+1/2)}{\sqrt{\pi}(\nu+1)^{\lambda+1}} \Big\{|\cot \varphi|+
		\frac{2}{3}\frac{\lambda+1}{\nu+1}\\
		&\quad +\frac{ 2^{2-\lambda} \Gamma(2\lambda-1)}{\Gamma(\lambda+1)} \frac{(\nu+1)^{\lambda+1}}{(\nu-\lambda+3)^{2\lambda-1}}|\cot \varphi|^{\lambda-2}
		\Big( |\cot \varphi|+\frac{2}{3}\frac{2\lambda-1}{\nu-\lambda+3}\Big)\Big\}.
		\end{split}
		\end{equation}
	\end{itemize}
	\end{thm}

\vskip 4pt

 With Theorem \ref{AsympGeg-A} at our disposal,  we next estimate the  bound of
 ${\mathcal  S}_\nu^{(\lambda)} (\varphi)$,  and characterize its explicit dependence of $\theta$ and  decay rate in $\nu.$
\begin{cor}
	\label{AsympGeg-B}
For  $\lambda>0,$  we have
\begin{equation}\label{AsympGeg-B0}
\Big|(\sin \varphi)^{\lambda}\,{}^{r\!}G_\nu^{(\lambda)}(\cos \varphi)-
	\frac{2^\lambda\Gamma(\lambda+1/2)}{\sqrt{\pi} {(\nu+\lambda)^{\lambda}}} \, {\cos ((\nu+\lambda)\varphi- \lambda\pi/2)} \Big|\le
	\frac{B^{(\lambda)}_\nu}
	{\nu^{\lambda+1}\sin \varphi}\,,
\end{equation}
where  the constant $B^{(\lambda)}_\nu$ is  given by
	\begin{itemize}
		\item[(i)] for $0<\lambda\le 2$ and $\nu+\lambda>1,$ 
		\begin{equation}\label{Bnuv1}
		{B}_\nu^{(\lambda)}=\frac{\lambda\, |\lambda-1|\,2^\lambda\,\Gamma(\lambda+1/2)}{\sqrt{\pi}}\,
\frac{3\nu+5\lambda-1}{3(\nu+\lambda-1)}\,
		{\rm exp}\Big( \frac{1-\lambda^2} {\nu+\lambda-1} \Big),
		\end{equation}
and the bound \eqref{AsympGeg-B0} holds for all $\theta\in (0,\pi);$
\vskip 4pt
		\item[(ii)] for $\lambda>  2$ and $\nu>\lambda-3,$  we have
				\begin{equation}\label{Bnuv2}
				\begin{split}	
	{B}_\nu^{(\lambda)}&=
	\frac{\lambda(\lambda-1)2^{3 \lambda/2}\Gamma(\lambda+1/2)}{3\sqrt{\pi}}\bigg\{\frac{3\nu+2\lambda+5}{\nu+1}\\
	&\quad +( c\pi)^{\lambda-2}\, \frac{ \Gamma(2\lambda-1)}{\Gamma(\lambda+1)}\, \frac{3\nu+\lambda +7}{\nu-\lambda+3}\, {\rm exp}\Big(\frac{(2\lambda-5)(\lambda+1)}{\nu-\lambda+3}\Big)\bigg\},
	\end{split}\end{equation}
	and  the bound \eqref{AsympGeg-B0} holds for all   $\varphi\in [c\nu^{-1},\pi-c\nu^{-1}]$ with $c$ being a fixed positive constant.
	\end{itemize}
\end{cor}

We provide the derivation of the above bounds right after the proof of Theorem  \ref{AsympGeg-A}.  Note that in the second case: $\lambda>2,$ the bound is only available  for   $\varphi\in [c\nu^{-1},\pi-c\nu^{-1}]$ with some fixed constant $c>0$.
 Indeed,  the situation is reminiscent to  the classical Gengenbauer polyomial with
  asymptotics only valid for $\varphi\in [cn^{-1},\pi-cn^{-1}]$ with large $n,$ as we  remark below.
\begin{rem} \label{remAsymptotic} {\em
From  \eqref{obsvers0} and Theorem {\rm \ref{AsympGeg-A}}, we obtain  that for $\nu=n\in {\mathbb N},$
\begin{equation}\label{remAsymptotic-0}
\begin{split}
& (\sin\theta)^{\lambda}  P_n^{(\lambda-1/2,\lambda-1/2)}(\cos \varphi)=(\sin\theta)^{\lambda}\,P_n^{(\lambda-1/2,\lambda-1/2)}(1)\,G_n^{(\lambda)}(\cos \varphi)\\
&\quad=\frac{2^\lambda\Gamma(n+\lambda+1/2)}{\sqrt{\pi} \,n!\,{(n+\lambda)^{\lambda}}} \, {\cos ((n+\lambda)\varphi- \lambda\pi/2)} +\frac{\Gamma(n+\lambda+1/2)}{\Gamma(\lambda+1/2)n!}{\mathcal  R}_n^{(\lambda)} (\varphi).
\end{split}
\end{equation}
Then  from  Corollary {\rm \ref{AsympGeg-B}},  we can derive the bounds uniform for $n.$ In fact, we can recover the asymptotic formula for the classical Gegenbauer polynomial with large $n$
{\rm(cf.} \cite[Thm 8.21.13]{Szego1975Book}{\rm):}
\begin{equation}\label{remAsymptotic-7}
(\sin\theta)^{\lambda}  P_n^{(\lambda-1/2,\lambda-1/2)}(\cos \varphi)=\frac{2^{\lambda}}{\sqrt{\pi n}}
\big\{\!\cos\big ((n+\lambda)\varphi -  \lambda\pi/2\big) + \big(n\sin\varphi\big)^{-1}O(1)\big\},
\end{equation}
 for all $\lambda>0$ and $\varphi\in [cn^{-1},\pi-cn^{-1}]$ with $n\gg 1$ and $c$ being  a fixed positive constant. Indeed,
using the property of the Gamma function {\rm(cf.} \cite[(6.1.38)]{Abramowitz1972Book}{\rm):}
\begin{equation}\label{PropertyGamma}
\Gamma(x+1)=\sqrt{2\pi}\,
x^{x+1/2}\exp\Big(\!-x+\frac{\eta}{12x}\Big),\quad
x>0,\;\;0<\eta<1,
\end{equation}
and the bounds of ${\mathcal  R}_n^{(\lambda)} (\varphi)$ in Corollary  {\rm \ref{AsympGeg-B}}, we can deduce \eqref{remAsymptotic-7}  straightforwardly. }
\end{rem}

\vskip 5pt
Thanks to Theorem \ref{AsympGeg-A}, we can derive the following uniform bounds for  $\theta\in[0,\pi],$ and nearly all fractional degree $\nu>0.$ We refer to Subsection \ref{proof-Thm2.2} for its proof.
\begin{thm}
	\label{AsympGeg-CaseA}
{\rm (i)}  If  $0<\lambda\le 2,\, \nu+\lambda>1$ and $\nu>0,$ we have
\begin{equation}\label{AsympGeg-cA}
|\widetilde {\mathcal  R}_\nu^{(\lambda)} (\varphi)|=\Big|(\sin \varphi)^{\lambda+1}\,{}^{r\!}G_\nu^{(\lambda)}(\cos \varphi)-
	\frac{2^\lambda\Gamma(\lambda+1/2)}{\sqrt{\pi} {(\nu+\lambda)^{\lambda}}} \, (\sin \theta)\, {\cos ((\nu+\lambda)\varphi- \lambda\pi/2)} \Big|\le
{\widetilde {\mathcal S}^{(\lambda)}_\nu}(\theta),
\end{equation}
for all $\theta\in [0,\pi],$ where ${\widetilde {\mathcal R}^{(\lambda)}_\nu}(\theta)=(\sin \theta)\,   {\mathcal R}_\nu^{(\lambda)} (\varphi)$ and ${\widetilde {\mathcal S}^{(\lambda)}_\nu}(\theta)=(\sin \theta)\,   {\mathcal S}_\nu^{(\lambda)} (\varphi).$

\medskip
{\rm (ii)}  If $\lambda>  2,\, \nu>\lambda-3$ and $\nu>0,$  we have
\begin{equation}\label{AsympGeg-cB}
|\widetilde {\mathcal  R}_\nu^{(\lambda)} (\varphi)|=\Big|(\sin \varphi)^{2\lambda-1}\,{}^{r\!}G_\nu^{(\lambda)}(\cos \varphi)-
	\frac{2^\lambda\Gamma(\lambda+1/2)}{\sqrt{\pi} {(\nu+\lambda)^{\lambda}}} \, (\sin \theta)^{\lambda-1} {\cos ((\nu+\lambda)\varphi- \lambda\pi/2)} \Big| \le
{\widetilde {\mathcal S}^{(\lambda)}_\nu}(\theta),
\end{equation}
for all $\theta\in [0,\pi],$ where ${\widetilde {\mathcal R}^{(\lambda)}_\nu}(\theta)=(\sin \theta)^{\lambda-1}\,   {\mathcal R}_\nu^{(\lambda)} (\varphi)$ and  ${\widetilde {\mathcal S}^{(\lambda)}_\nu}(\theta)=(\sin \theta)^{\lambda-1}\,   {\mathcal S}_\nu^{(\lambda)} (\varphi).$
\end{thm}

\begin{figure}[!th]
	\begin{center}
\includegraphics[width=0.41\textwidth,height=0.3\textwidth]{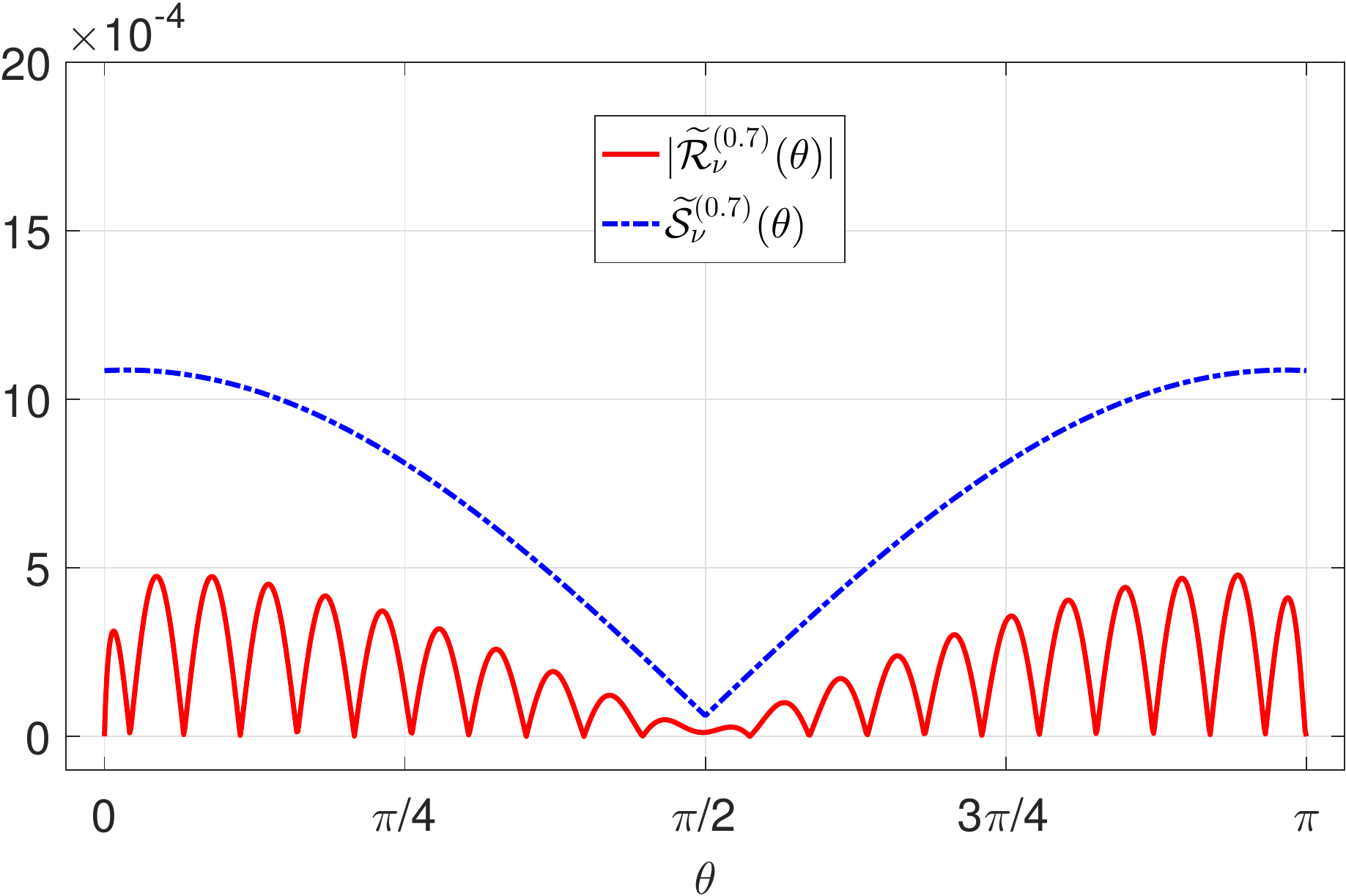}\qquad
		\includegraphics[width=0.41\textwidth,height=0.3\textwidth]{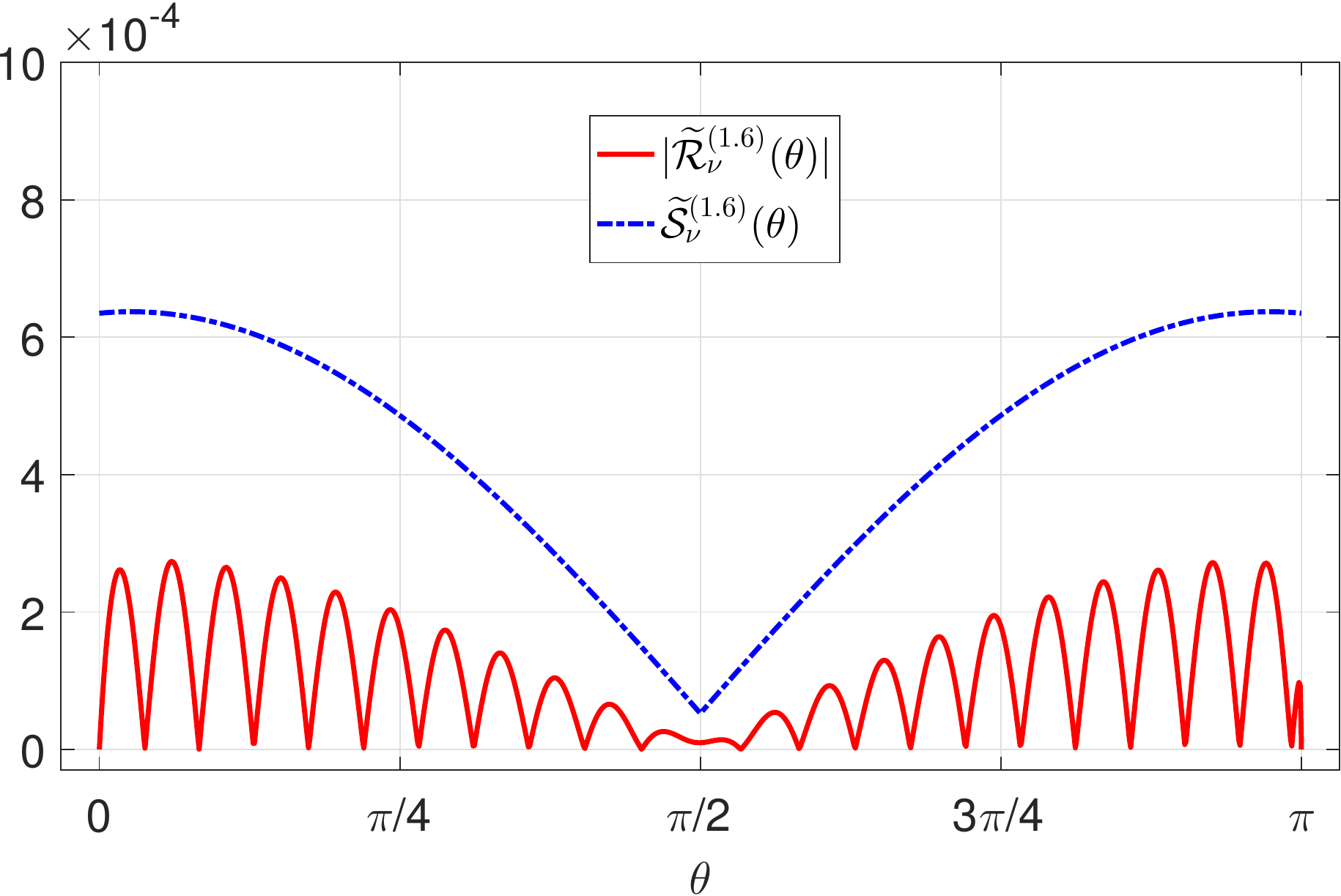}\\
\includegraphics[width=0.41\textwidth,height=0.3\textwidth]{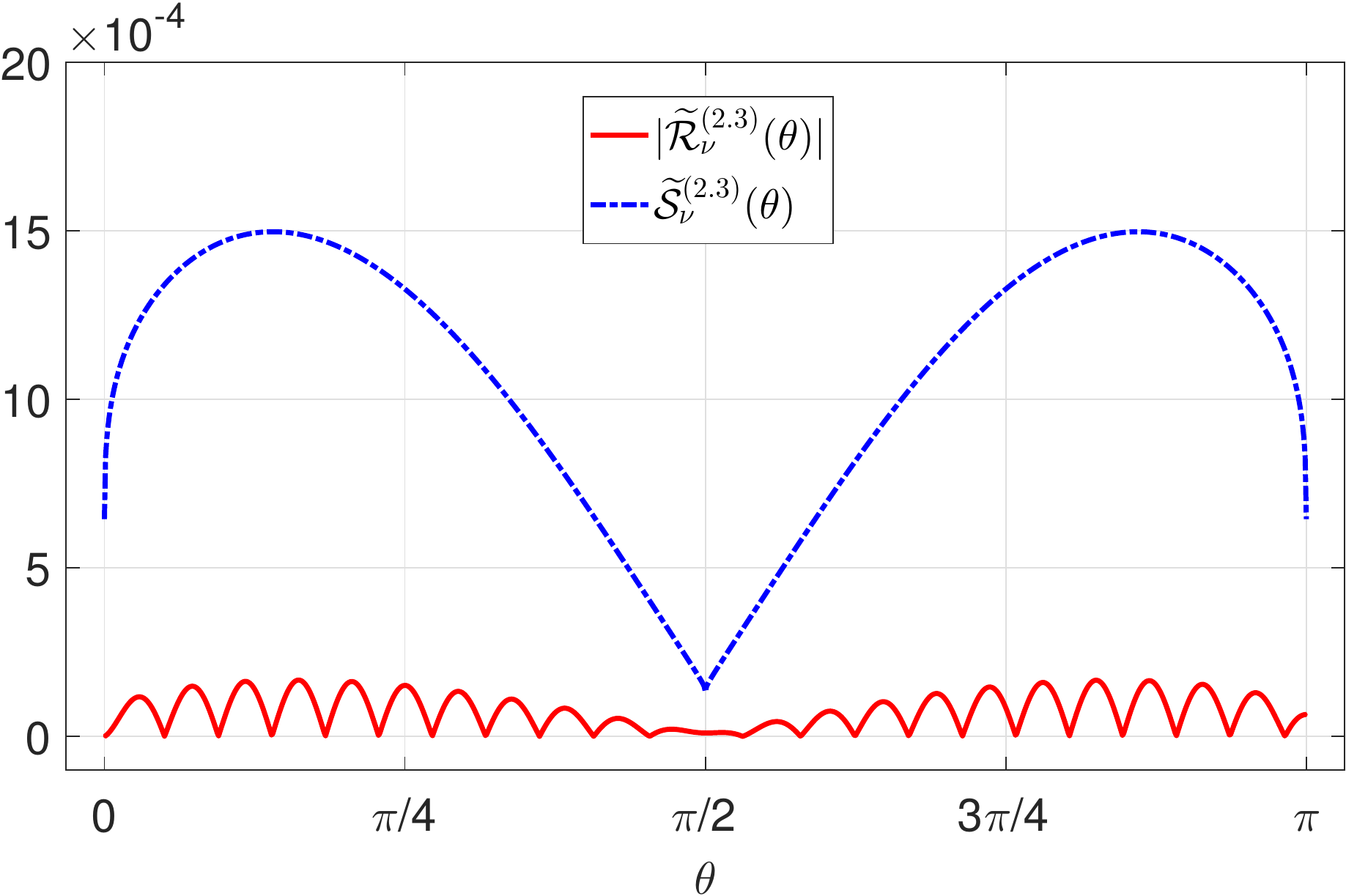}\qquad
		\includegraphics[width=0.41\textwidth,height=0.3\textwidth]{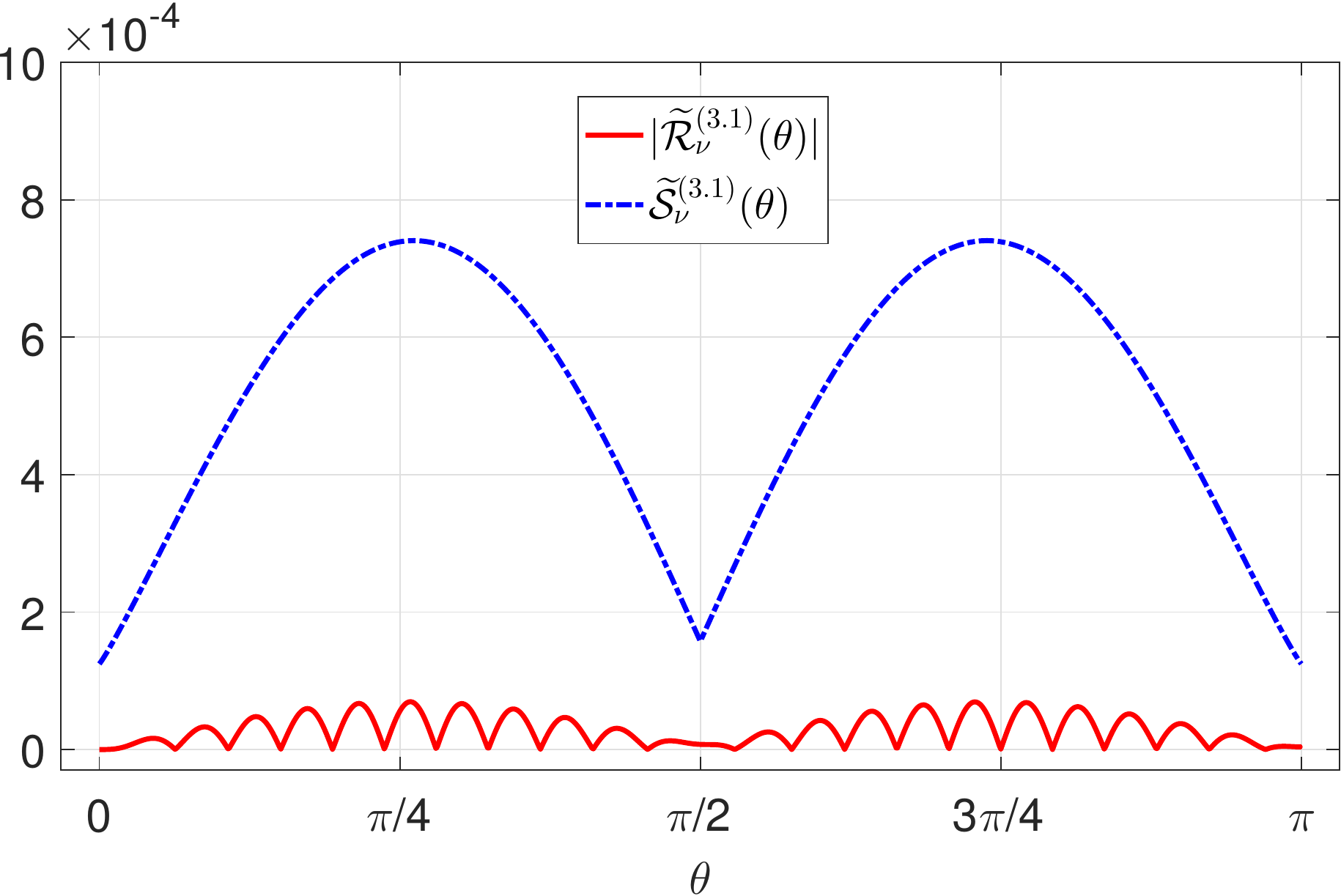}
		\caption{Plots of $|\widetilde {\mathcal  R}_\nu^{(\lambda)} (\varphi)|$  and $\widetilde {\mathcal  S}_\nu^{(\lambda)} (\varphi)$ in Theorem \ref{AsympGeg-CaseA}, where $\theta\in [0,\pi], \lambda=0.7,1.6, 2.3, 3.1$ and  $\nu=20.3$.}
		\label{FigForBound}
	\end{center}
\end{figure}

In the end of this section, we provide some numerical illustrations of the unform bounds in
Theorem \ref{AsympGeg-CaseA}. In  Figure \ref{FigForBound},
we plot the graphs of
$|\widetilde {\mathcal  R}_\nu^{(\lambda)} (\varphi)|$  and $\widetilde {\mathcal  S}_\nu^{(\lambda)} (\varphi)$  for $\theta\in [0,\pi]$ and with $\lambda=0.7,1.6, 2.3, 3.1,$  $\nu=20.3$. Indeed, we observe that in all cases, the curves of the upper bounds are on the top of $|\widetilde {\mathcal  R}_\nu^{(\lambda)} (\varphi)|,$ and ``sharp" corner of $\widetilde {\mathcal  S}_\nu^{(\lambda)} (\varphi)$ at $\theta=\pi/2$ is largely due to the involved $|\cos\theta|.$

\section{Proof of the  results}\label{sect:proofs}
\setcounter{equation}{0}
\setcounter{lmm}{0}
\setcounter{thm}{0}

\subsection{Two lemmas} As the proof of the main result is quite involved, we   take several steps and summarise the intermediate results into two lemmas.

\begin{lmm} \label{UperBoundf} For  real $\lambda>0,$ $\varphi\in (0,\pi)$  and  $t>0,$  define
\begin{equation}\label{gttf}
\begin{split}
& g(\varphi, t):= \frac{\cos(\varphi-\ri t) - \cos\varphi}{t}= \frac{\cos\varphi\, (\cosh t-1)+ \ri \sin \varphi\, \sinh t}{t},\\
& f^{(\lambda)}(\varphi, t):=\frac{g^{\lambda-1}(\varphi, t)-	g^{\lambda-1}(\varphi, 0)} t,\quad g(\varphi, 0):=\lim_{t\to 0^+} g(\varphi, t)=\ri \sin \varphi.
\end{split}
\end{equation}
Then we have for $\varphi\in (0,\pi)$ and $t>0,$
\begin{itemize}
	\item[(i)] for $0<\lambda\le 2,$
	\begin{equation}\label{ftbnds}
	|f^{(\lambda)}(\varphi,t)| \le |\lambda-1| \, (\sin \varphi)^{\lambda-1} \,  \Big( |\cot \varphi| +\frac {2t} 3 \Big) e^t;
	\end{equation}
	\item[(ii)] for $\lambda> 2,$
	\begin{equation}\label{ftbndsB}
	|f^{(\lambda)}(\varphi,t)| \le  2^{\lambda/2}\, (\lambda-1)  \,  (\sin \varphi)^{\lambda-1} \,  \Big( |\cot \varphi| +\frac {2t} 3 \Big) \Big(1+
	\frac{|\cot \varphi|^{\lambda-2}} {2^{\lambda-2}}\, t^{\lambda-2} e^{(\lambda-2)t}\Big) e^{(\lambda-1)t}.
	\end{equation}
\end{itemize}
\end{lmm}
To avoid distracting from proving the main result, we put this a bit lengthy  proof but only involving fundamental calculus  in Appendix \ref{AppendixA}.

A critical step is to show that the integral in  \eqref{IntRep}  satisfies the following identity.
\begin{lemma}\label{indtity} For real $\nu,\lambda\ge 0,$ and $\theta\in (0,\pi),$ we have
\begin{equation}\label{mainStep1}
\int^\varphi_0\frac{\cos((\nu+\lambda)\vartheta)}{(\cos \vartheta-\cos\varphi)^{1-\lambda}}\,d\vartheta=\frac{\Gamma(\lambda)} {(\nu+\lambda)^{\lambda}}\,
\frac{\cos ((\nu+\lambda)\varphi- \lambda\pi/2)}{	(\sin \varphi)^{1-\lambda}} +\breve{\mathcal  R}_\nu^{(\lambda)} (\varphi),
\end{equation}
where
\begin{equation}\label{ResidualDefn}
\begin{split}
& \breve {\mathcal  R}_\nu^{(\lambda)} (\varphi):= \int^{ \infty}_0 \RE\big\{\ri\, e^{-\ri (\nu+\lambda)\varphi}	f^{(\lambda)}(\varphi, t)\big\}t^{\lambda}e^{-(\nu+\lambda) t}dt,
\end{split}
\end{equation}
and $f^{(\lambda)}(\varphi, t)$ is defined in   \eqref{gttf}.
\end{lemma}

\begin{proof} It  is evident that  by the parity, we have
\begin{equation}\label{residual}
\int^\varphi_0\frac{\cos((\nu+\lambda)\vartheta)}{(\cos \vartheta-\cos\varphi)^{1-\lambda}}\,d\vartheta=\frac 1 2 \int^{\varphi}_{-\varphi}
 F_\nu^{(\lambda)} (\varphi, \vartheta)\, d\vartheta\,,
\end{equation}
where we denote
\begin{equation}\label{newfuncs}
 F_\nu^{(\lambda)} (\varphi, \vartheta):=
\frac {e^{\ri (\nu+\lambda)\vartheta}} {(\cos \vartheta-\cos\varphi)^{1-\lambda}}\,.
\end{equation}

\vskip 5pt
We consider the cases with $\lambda\ge 1$ and $0<\lambda<1,$ separately.
\vskip 4pt
\noindent\underline{{\bf (i)}  Proof of  \eqref{mainStep1} with   $\lambda\ge 1$.}~
From the Cauchy-Goursat theorem, we infer that  for  any fixed $\varphi\in (0,\pi)$ and real $\nu> 0,$  the contour integration of  $F_\nu^{(\lambda)} (\varphi, \cdot)$ (with  an extension to the complex plane) along the rectangular contour  in Figure \ref{FigForContour} (left), is zero.
 Thus, we have
	\begin{equation}\label{IntRep-2}
	\begin{split}
\int^{\varphi}_{-\varphi}
 F_\nu^{(\lambda)} (\varphi, \vartheta)\, d\vartheta&
	=  \int^{-\varphi+\ri R}_{-\varphi}  F_\nu^{(\lambda)} (\varphi, \vartheta)\, d\vartheta - \int^{\varphi+\ri R}_{\varphi}  F_\nu^{(\lambda)} (\varphi,\vartheta) \, d\vartheta+ \int^{\varphi+\ri R}_{-\varphi+\ri R}  F_\nu^{(\lambda)} (\varphi, \vartheta) \,d\vartheta\\
	& = \ri \int_0^R \big\{
	 F_\nu^{(\lambda)} (\varphi,  -\varphi+\ri t) - F_\nu^{(\lambda)} (\varphi,  \varphi+\ri t)\big\} dt + \int^{\varphi}_{-\varphi}
	 F_\nu^{(\lambda)} (\varphi,  t+\ri R) \,dt, 
\end{split}\end{equation}	
where we made the change of variables for three integrals: $\vartheta= -\varphi+\ri t, \varphi+\ri t, t+\ri R,$ respectively.

	\vskip 5pt
\begin{figure}[!ht]
	\begin{center}
		{~}\hspace*{-20pt}	\includegraphics[width=0.4\textwidth, height=0.22\textwidth]{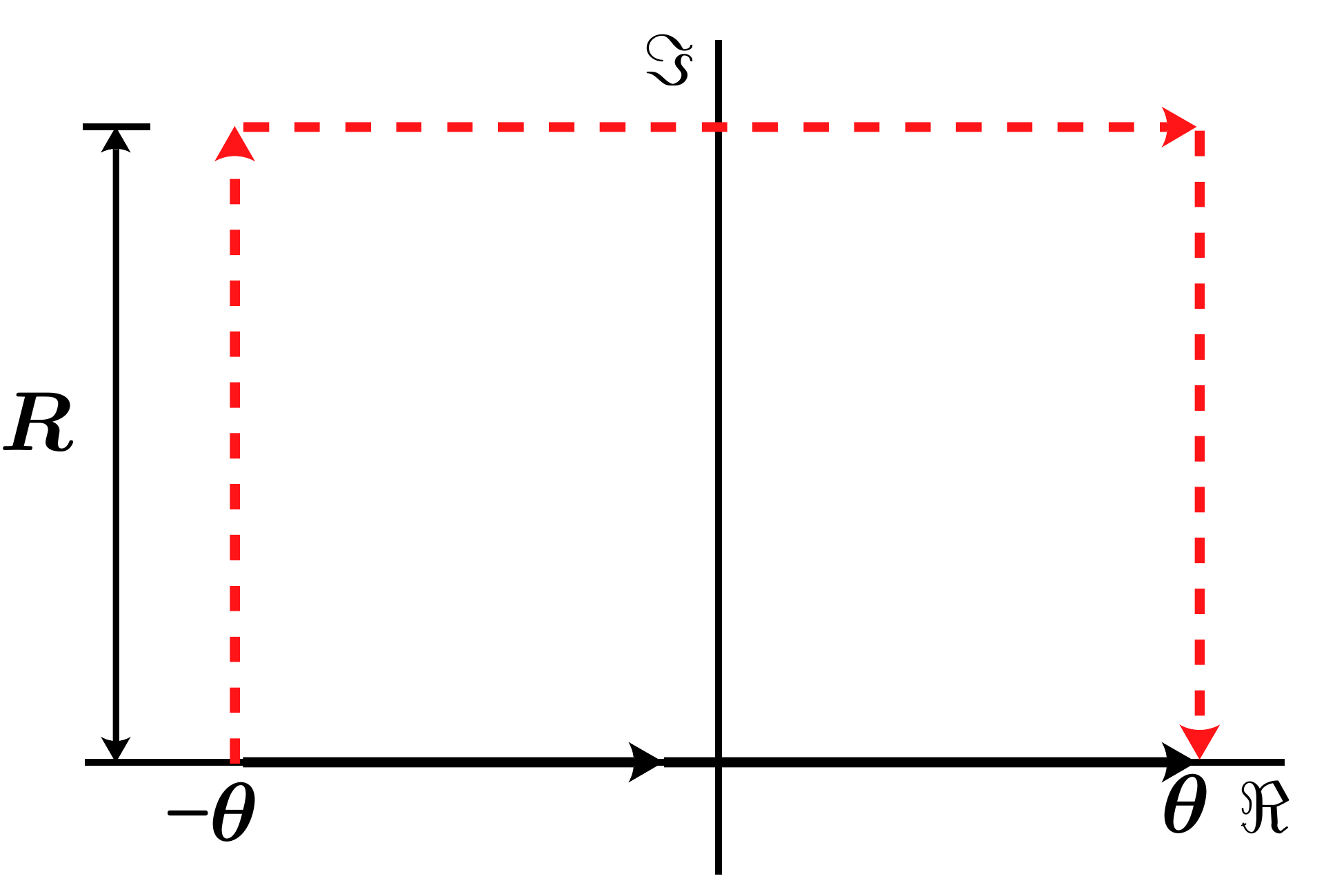}\quad \qquad
		\includegraphics[width=0.4\textwidth, height=0.22\textwidth]{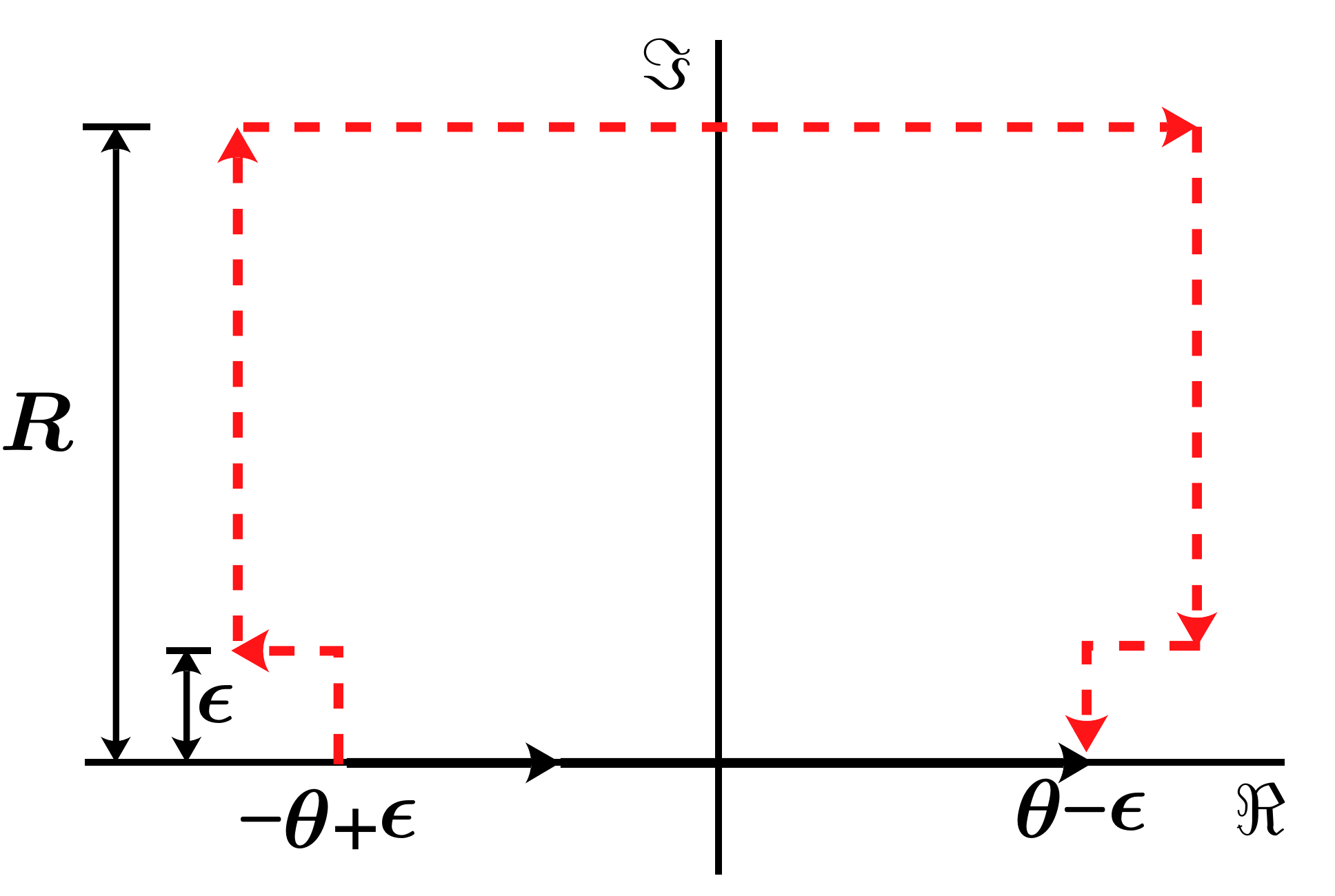}
		\caption{Contour integral for \eqref{IntRep-2}. Left: for $\lambda\ge 1$; Right: for $0<\lambda<1$.}
		\label{FigForContour}
	\end{center}
\end{figure}

For $\lambda\ge 1$ and $R>0,$ we have
\begin{equation}\label{IntRep-3}
\begin{split}
|F_\nu^{(\lambda)} (\varphi,  t+\ri R)|&=\frac{|e^{\ri (\nu+\lambda)t- (\nu+\lambda)R}|}{|\cos (t+\ri R)-\cos\varphi|^{1-\lambda}} \\
&= e^{-(\nu+\lambda) R} \big( (\cos t\cosh R-\cos \varphi)^2+\sin^2t\sinh^2R \big)^{(\lambda-1)/2}\\
&\le  e^{-(\nu+\lambda) R} \big( (\cosh R+1)^2+\sinh^2R \big)^{(\lambda-1)/2}\\
&= 2^{(1-\lambda)/2} e^{-(\nu+1)R} \big(1+2e^{-R}+ 2e^{-2R}+2e^{-3R}+e^{-4R}\big)^{(\lambda-1)/2}\\
&<2^{\lambda-1}  e^{-(\nu+1)R}.
	\end{split}
	\end{equation}
Thus, we have
		\begin{equation}\label{IntRep-3+1}
	\lim_{R\to\infty} \int^{\varphi+\ri R}_{-\varphi+\ri R}  F_\nu^{(\lambda)} (\varphi, \vartheta) \,d\vartheta=\lim_{R\to\infty}  \int^{\varphi}_{-\varphi}
	 F_\nu^{(\lambda)} (\varphi,  t+\ri R) \,dt=0.
	\end{equation}

Recall the notation in \eqref{gttf}: $tg(\varphi, t)=\cos(\varphi-\ri t) - \cos\varphi.$
In view of \eqref{gttf}, we can write $g^{\lambda-1} (\varphi,t)=g^{\lambda-1} (\varphi,0)+t f^{(\lambda)}(\varphi,t).$
Thus, by a direct calculation, we obtain
\begin{equation}\label{IntRep-3007}
\begin{split}
		\ri F_\nu^{(\lambda)}  (\varphi,  -\varphi+\ri t)  &- \ri F_\nu^{(\lambda)} (\varphi,  \varphi+\ri t)\\
		&=\big\{\ri\, g^{\lambda-1} (\varphi,t) e^{-\ri (\nu+\lambda)\varphi}+ \big(\ri\, g^{\lambda-1} (\varphi,t) e^{-\ri (\nu+\lambda)\theta}\big)^* \big\} t^{\lambda-1} e^{-(\nu+\lambda)t}\\
		&= 2\, \RE \big\{\ri\, g^{\lambda-1} (\varphi,t) e^{-\ri (\nu+\lambda)\varphi}\big\} t^{\lambda-1} e^{-(\nu+\lambda)t}\\
		&= 2\, \RE \big\{\ri\, g^{\lambda-1} (\varphi,0) e^{-\ri (\nu+\lambda)\varphi}\big\} t^{\lambda-1} e^{-(\nu+\lambda)t}
		\\&\quad + 2\, \RE \big\{\ri\, f^{(\lambda)} (\varphi,t) e^{-\ri (\nu+\lambda)\varphi}\big\} t^{\lambda} e^{-(\nu+\lambda)t}.
\end{split}
\end{equation}	
Since $\ri =e^{\ri \pi/2}$ and $g(\varphi,0)=(\ri \sin\varphi)^{\lambda-1}$, we have
	\begin{equation}\begin{split}\label{IntRep-6+30}
	\RE \big\{\ri\, g^{\lambda-1} (\varphi,0) e^{-\ri (\nu+\lambda)\varphi}\big\}&=\RE\big\{\ri\,(\ri \sin \varphi)^{\lambda-1}\, e^{-\ri (\nu+\lambda)\varphi}\big\}=
	(\sin \varphi)^{\lambda-1}\,\RE\big\{e^{-\ri ((\nu+\lambda)\varphi-\lambda\pi/2)}\big\}\\
	&=
	(\sin \varphi)^{\lambda-1}\cos ((\nu+\lambda)\varphi- \lambda\pi/2).
	\end{split}\end{equation}
	Using the definition of the Gamma function, we find  that for any $a>0, z>-1,$
	\begin{equation}\begin{split}\label{GGamma}	
	\int^{ \infty}_0 t^{z}e^{-a t}dt=\frac{\Gamma(z+1)}{a^{z+1}},\;\;\;  {\rm as}\;\;\;  \Gamma(z+1)=\int^\infty _0e^{-t}t^zdt.
	\end{split}\end{equation}
As a direct consequence of \eqref{IntRep-6+30}-\eqref{GGamma},  we have
\begin{equation}\label{newReA}
 2\, \int_{0}^\infty \RE \big\{\ri\, g^{\lambda-1} (\varphi,0) e^{-\ri (\nu+\lambda)\varphi}\big\} t^{\lambda-1} e^{-(\nu+\lambda)t} dt
 =\frac{2\, \Gamma(\lambda)} {(\nu+\lambda)^{\lambda}}\,
\frac{\cos ((\nu+\lambda)\varphi- \lambda\pi/2)}{	(\sin \varphi)^{1-\lambda}}.
\end{equation}

Letting $R\to\infty$ in  \eqref{IntRep-2},  we obtain \eqref{mainStep1}-\eqref{ResidualDefn} from \eqref{residual},   \eqref{IntRep-3+1}-\eqref{IntRep-3007} and
\eqref{newReA} directly.

\vskip 10pt
\noindent\underline{{\bf (ii)}   Proof of  \eqref{mainStep1} with  $0<\lambda< 1$.}~  In this case, we integrate along a similar contour but exclude  singular points $\vartheta=\pm \varphi,$ as depicted in  Figure \ref{FigForContour} (right), where  $0<\epsilon<\varphi.$
Like  \eqref{IntRep-2}, we have
 \begin{equation}\label{IntRep-2-1}
\begin{split}
\int^{\varphi-\epsilon}_{-\varphi+\epsilon}
F_\nu^{(\lambda)} (\varphi,& \vartheta)\, d\vartheta
=  I_1(\epsilon,R)+I_2(R)+ I_3(\epsilon)+I_4(\epsilon),
\end{split}\end{equation}
where
 \begin{equation}\label{IntRep2107}
\begin{split}
& I_1(\epsilon,R):=\int^{-\varphi+\ri R}_{-\varphi+\ri \epsilon}  F_\nu^{(\lambda)} (\varphi, \vartheta)\, d\vartheta
- \int^{\varphi+\ri R}_{\varphi+\ri \epsilon}  F_\nu^{(\lambda)} (\varphi, \vartheta) \, d\vartheta,\\
& I_2(R):=\int^{\varphi+\ri R}_{-\varphi+\ri R}  F_\nu^{(\lambda)} (\varphi, \vartheta) \,d\vartheta,
\\&  I_3(\epsilon):= \int_{-\varphi+\epsilon} ^ {-\varphi+\epsilon+\ri \epsilon} F_\nu^{(\lambda)} (\varphi, \vartheta) \,d\vartheta -\int_{\varphi-\epsilon} ^ {\varphi-\epsilon+\ri \epsilon} F_\nu^{(\lambda)} (\varphi, \vartheta) \,d\vartheta\\
&  I_4(\epsilon):= -\int_{-\varphi+\ri \epsilon} ^ {-\varphi+\epsilon+\ri \epsilon} F_\nu^{(\lambda)} (\varphi, \vartheta) \,d\vartheta
  -\int_{\varphi-\epsilon+\ri \epsilon} ^ {\varphi+\ri \epsilon} F_\nu^{(\lambda)} (\varphi, \vartheta) \,d\vartheta.
\end{split}\end{equation}

Using  a change of variable: $\vartheta= \pm\varphi+\ri t,$ and noting that  the derivation in  \eqref{IntRep-3007}-\eqref{IntRep-6+30} is valid for $0<\lambda<1,$ we have
\begin{equation}\label{AcaseB}
\begin{split}
 I_1(\epsilon,R)&= \ri \int_\epsilon^R \big\{
	 F_\nu^{(\lambda)} (\varphi,  -\varphi+\ri t) - F_\nu^{(\lambda)} (\varphi,  \varphi+\ri t)\big\} dt\\
	 &= 2\frac{\cos ((\nu+\lambda)\varphi- \lambda\pi/2)}{ (\sin \varphi)^{1-\lambda}} \int_\epsilon^R t^{\lambda-1} e^{-(\nu+\lambda)t} dt \\
	 &\quad + 2\, \int_\epsilon^R \RE \big\{\ri\, f^{(\lambda)} (\varphi,t) e^{-\ri (\nu+\lambda)\varphi}\big\} t^{\lambda} e^{-(\nu+\lambda)t} dt.
	 \end{split}
\end{equation}
From  \eqref{ftbnds} and \eqref{GGamma}	-\eqref{newReA},   we infer that
\begin{equation}\label{limiI1}
\begin{split}
\lim_{\epsilon\to 0; R\to\infty} I_1(\epsilon,R)&= \frac{2\, \Gamma(\lambda)} {(\nu+\lambda)^{\lambda}}\,
\frac{\cos ((\nu+\lambda)\varphi- \lambda\pi/2)}{	(\sin \varphi)^{1-\lambda}}\\
&\quad +
 2\, \int_0^\infty \RE \big\{\ri\, f^{(\lambda)} (\varphi,t) e^{-\ri (\nu+\lambda)\varphi}\big\} t^{\lambda} e^{-(\nu+\lambda)t} dt.
 \end{split}
\end{equation}
Therefore, it suffices to show
\begin{equation}\begin{split}\label{IntRep-2+1}
\lim_{R\to \infty}I_2(R)
=0,\quad \lim_{\epsilon\to 0} I_3(\epsilon)=\lim_{\epsilon\to 0} I_4(\epsilon)=0.
\end{split}\end{equation}	

By \eqref{IntRep-2},  we have
\begin{equation}\label{I2bnde}
 I_2(R)=\int^{\varphi}_{-\varphi}
	 F_\nu^{(\lambda)} (\varphi,  t+\ri R) \,dt,
\end{equation}
and
\begin{equation}\label{IntRep-3baba}
\begin{split}
|F_\nu^{(\lambda)} (\varphi,  t+\ri R)|&=\frac{|e^{\ri (\nu+\lambda)t- (\nu+\lambda)R}|}{|\cos (t+\ri R)-\cos\varphi|^{1-\lambda}}\\
&= e^{-(\nu+\lambda) R} \big( (\cos t\cosh R-\cos \varphi)^2+\sin^2t\sinh^2R \big)^{(\lambda-1)/2}\\
&\le  e^{-(\nu+\lambda) R} (\sinh R) ^{\lambda-1} |\sin t|^{\lambda-1}.
	\end{split}
	\end{equation}
Thus, for $0<\lambda<1$ and $\varphi\in (0,\pi),$
\begin{equation}\label{I2bnds}
|I_2(R)|\le  \frac{2 e^{-(\nu+\lambda) R}}
{(\sinh R) ^{1-\lambda}} \int_0^\varphi \frac{1} {(\sin t)^{1-\lambda}} dt\to 0,\;\;\; {\rm as}\;\; R\to \infty.
\end{equation}

Next, using a change of variable: $\theta=-\varphi+ \epsilon +\ri t,  \varphi-\epsilon +\ri t,$ respectively, for two integrals, we obtain from
a direct calculation that
\begin{equation}\label{IntRep335}
\begin{split}
I_3(\epsilon)&= \ri \int_{0} ^ {\epsilon} \big\{F_\nu^{(\lambda)} (\varphi,- \varphi+ \epsilon +\ri t) -F_\nu^{(\lambda)} (\varphi,\varphi- \epsilon +\ri t) \big\} dt\\
&=2 \int_0^\epsilon \Re\Big\{\frac{\ri e^{-\ri (\nu+\lambda)(\varphi-\epsilon)}}{(\cos(\varphi-\epsilon-\ri t)-\cos\varphi)^{1-\lambda}} \Big\}  e^{-(\nu+\lambda)t} dt.
\end{split}
\end{equation}
Note that we have
\begin{equation}\label{IntRep469}
\begin{split}
|\cos(\varphi-\epsilon-\ri t)-\cos\varphi|&=((\cos(\varphi-\epsilon)\cosh t-\cos \varphi)^2+\sin^2(\varphi-\epsilon)\sinh^2 t)^{1/2}\\
&\ge |\sin (\varphi-\epsilon)||\sinh t| \ge |\sin (\varphi-\epsilon)||\sin t|,
\end{split}
\end{equation}
 where we used the inequality: $|\sin t|\le \sinh t$ for $t>0$ (cf. \cite[(4.18.9)]{Olver2010Book}). Therefore, for $0<\lambda<1,$ we have
\begin{equation}\label{I3bnds}
|I_3(\epsilon)|\le  \frac{2 }
{(\sin (\varphi-\epsilon)) ^{1-\lambda}} \int_0^\epsilon {(\sin t)^{\lambda-1}} dt\to 0,\;\;\; {\rm as}\;\; \epsilon\to 0.
\end{equation}


Similarly, with a change of variable: $\theta=-\varphi+\ri \epsilon + t,  \varphi+ \ri \epsilon - t,$ respectively, for two integrals,
\begin{equation}\label{IntRep187}
\begin{split}
I_4(\epsilon)&= - \int_{0} ^ {\epsilon} \big\{F_\nu^{(\lambda)} (\varphi, -\varphi+\ri \epsilon + t) +F_\nu^{(\lambda)} (\varphi, \varphi+ \ri \epsilon - t) \big\} dt\\
&= -2 e^{-\epsilon(\nu+\lambda)}\int_0^\epsilon \Re\Big\{\frac{e^{\ri (\nu+\lambda)(t-\varphi)}}{(\cos(t-\varphi+\ri \epsilon)-\cos\varphi)^{1-\lambda}} \Big\} dt.
\end{split}
\end{equation}
It is evident that
\begin{equation}\label{IntRep520}
\begin{split}
|\cos(t-\varphi+\ri\epsilon)-\cos\varphi|&=\big((\cos(t-\varphi)\cosh \epsilon-\cos \varphi)^2+\sin^2(t-\varphi)\sinh^2 \epsilon\big)^{1/2}\\
&\ge\sin(\varphi-t)\sinh \epsilon,
\end{split}
\end{equation}
where as  $0<\varphi<\pi$ and $0<t<\epsilon<\varphi$, we have
$$0<\varphi-\epsilon<\varphi-t<\varphi<\pi.$$
By  the fundamental inequalities,
\begin{equation}\label{RecallJordanIneqty}
\frac{2}{\pi} z\le \sin z\le z ,\quad
 z \in (0,\pi/2),
\end{equation}
	we obtain
\begin{equation}\label{IntRepPartI-1}
\frac{1}{\sin z}=\frac{1}{\sin(\pi-z)}\le
\frac{\pi}{2}\max\Big\{\frac{1}{z},\frac{1}{\pi-z}\Big\},\quad
z\in (0, \pi).
\end{equation}
This implies
\begin{equation}\label{IntRepPartI-2}
\sin(\varphi-t)\ge
\frac{2}{\pi}\min\big\{\varphi-t,\pi-\varphi+t\big\}> \frac{2}{\pi}\min\big\{\varphi-\epsilon,\pi-\varphi\big\}.
\end{equation}
From \eqref{IntRep187}	and \eqref{IntRepPartI-2}, we obtain
\begin{equation}\label{IntRepPartI-3}
\begin{split}
|I_4(\epsilon)|&
\le  2  e^{-\epsilon(\nu+\lambda)} \int_0^\epsilon |\cos(t-\varphi+\ri \epsilon)-\cos\varphi|^{\lambda-1} dt\\
&\le  \frac{2^{\lambda}}{\pi^{\lambda-1}}\,\frac{\epsilon}  {(\sinh \epsilon)^{1-\lambda}} \max\big\{(\varphi-\epsilon)^{\lambda-1} ,(\pi-\varphi)^{\lambda-1} \big\}\to 0,  \;\; {\rm as}\;\; \epsilon\to 0.
\end{split}
\end{equation}

Thus, letting $\epsilon\to 0$ and $R\to \infty$ in \eqref{IntRep-2-1}, we obtain  \eqref{mainStep1}-\eqref{ResidualDefn} with $0<\lambda<1$   from \eqref{residual},  \eqref{limiI1}, \eqref{I2bnds},  \eqref{I3bnds} and \eqref{IntRepPartI-3}.
\end{proof}


\subsection{Proof of Theorem \ref{AsympGeg-A}} With the bounds and identity  in Lemmas \ref{UperBoundf}-\ref{indtity}, we are ready to  show  the main result.

From  \eqref{IntRep} and  Lemma \ref{indtity},  we derive
\begin{equation}\label{ResidualRnu}
\begin{split}
	{\mathcal  R}_\nu^{(\lambda)} (\varphi) &= \frac{2^\lambda\,\Gamma(\lambda+1/2)}{\sqrt{\pi}\,  \Gamma(\lambda)}  (\sin\varphi)^{1-\lambda}\, \breve{\mathcal  R}_\nu^{(\lambda)} (\varphi)\\
	  &=\frac{2^\lambda\,\Gamma(\lambda+1/2)}{\sqrt{\pi}\, \Gamma(\lambda)}	(\sin \varphi)^{1-\lambda}  \int^{ \infty}_0 \RE\big\{\ri\, e^{-\ri (\nu+\lambda)\varphi}	 f^{(\lambda)}(\varphi, t)\big\}t^{\lambda}e^{-(\nu+\lambda) t}dt.
\end{split}
\end{equation}
We now estimate $\breve {\mathcal R}_\nu^{(\lambda)}(\varphi)$ in  \eqref{mainStep1}-\eqref{ResidualDefn}   by using Lemma \ref{UperBoundf}. 

(i) For $0<\lambda\le 2$  and $\nu+\lambda>1,$ we obtain  from \eqref{ftbnds} and \eqref{GGamma}  that
\begin{equation}\label{IntRepPartII-1}
\begin{split}
|\breve{\mathcal  R}_\nu^{(\lambda)} (\varphi)|&
\le   \int^{ \infty}_0| 	f^{(\lambda)}(\varphi, t)|\, t^{\lambda}e^{-(\nu+\lambda) t}dt\\
&\le |\lambda-1| \, (\sin \varphi)^{\lambda-1} \, \int^{ \infty}_0  \Big( |\cot \varphi| +\frac {2t} 3 \Big)\, t^{\lambda}e^{-(\nu+\lambda-1) t}dt\\
&=\frac{|\lambda-1|\, \Gamma(\lambda+1)}{(\nu+\lambda-1)^{\lambda+1}}( \sin\varphi)^ {\lambda-1}\big( |\cot \varphi|+\frac{2}{3}\frac{\lambda+1}{\nu +\lambda-1}\Big).
\end{split}
\end{equation}

(ii) For $\lambda>  2$ and $\nu>\lambda-3,$ we derive from \eqref{ftbndsB} and \eqref{GGamma}  that
\begin{equation}\label{IntRepPartII-2}
\begin{split}
 |\breve{\mathcal  R}_\nu^{(\lambda)} & (\varphi)|\le    \int^{ \infty}_0| 	 f^{(\lambda)}(\varphi, t)|\, t^{\lambda}e^{-(\nu+\lambda) t}dt\\
&\le   2^{\lambda/2} (\lambda-1)   (\sin \varphi)^{\lambda-1} \int^{ \infty}_0 \Big( |\cot \varphi| +\frac {2t} 3 \Big) \Big(1+
\frac{|\cot \varphi|^{\lambda-2}} {2^{\lambda-2}}\, t^{\lambda-2} e^{(\lambda-2)t}\Big) t^{\lambda}e^{-(\nu+1) t}dt\\
&=\frac{2^{\lambda/2} (\lambda-1)  \,\Gamma(\lambda+1)}{(\nu+1)^{\lambda+1}} \,(\sin \varphi)^{\lambda-1} \Big\{|\cot \varphi|+
\frac{2}{3}\frac{\lambda+1}{
\nu+1}\\
&\quad +\frac{ 2^{2-\lambda} \Gamma(2\lambda-1)}{\Gamma(\lambda+1)} \frac{(\nu+1)^{\lambda+1}}{(\nu-\lambda+3)^{2\lambda-1}}|\cot \varphi|^{\lambda-2}
\Big( |\cot \varphi|+\frac{2}{3}\frac{2\lambda-1}{\nu-\lambda+3}\Big)\Big\}.
\end{split}
\end{equation}

Thanks to  \eqref{ResidualRnu},  we can derive the bounds in \eqref{BoundIntFunft-A}-\eqref{BoundIntFunft-B} from this relation and \eqref{IntRepPartII-1}-\eqref{IntRepPartII-2}, respectively.		
	This completes the proof of Theorem  \ref{AsympGeg-A}.


\subsection{Proof of Corollary \ref{AsympGeg-B}}  We prove two cases separately.

\vskip 3pt
(i) We obtain from \eqref{BoundIntFunft-A} that
	\begin{equation}\begin{split}\label{AsympGeg-B3}	
{\nu^{\lambda+1}\sin \varphi}\,  {| {\mathcal  R}_\nu^{(\lambda)} (\varphi)|}&\le
\frac{\lambda|\lambda-1|2^\lambda\Gamma(\lambda+1/2)}{\sqrt{\pi}} \Big(|\cos \varphi|+\frac 2 3 \frac{(\lambda+1) \sin \varphi} {\nu+\lambda-1}\Big)\Big(1+\frac{1-\lambda}{\nu+\lambda-1}\Big)^{\lambda+1}\\
&\le \frac{\lambda |\lambda-1|2^\lambda\Gamma(\lambda+1/2)}{\sqrt{\pi}} \Big(1+\frac 2 3 \frac{\lambda+1} {\nu+\lambda-1}\Big) \Big(1+\frac{1-\lambda}{\nu+\lambda-1}\Big)^{\lambda+1}.
	\end{split}\end{equation}
%
Using the basic inequality: $\ln(1+z)\le z$ for $z>-1,$ 
	we find 
	\begin{equation}\begin{split}\label{AsympGeg-B5}	
\Big(1+\frac{1-\lambda}{\nu+\lambda-1}\Big)^{\lambda+1}&={\rm exp}\Big({(\lambda+1)\ln\Big(1+\frac{1-\lambda}{\nu+\lambda-1} \Big) }\Big)\le {\rm exp}\Big( \frac{1-\lambda^2} {\nu+\lambda-1} \Big).
	\end{split}\end{equation}
Thus, we obtain  $B_\nu^{(\lambda)}$ immediately from the above for this case.

%
\vskip 5pt
(ii) For $\lambda> 2,$ $\nu-\lambda+3\ge 0$ and $\varphi\in [c\nu^{-1},\pi-c\nu^{-1}]$,  we obtain from \eqref{BoundIntFunft-B} that
	\begin{equation}\begin{split}\label{AsympGeg-B8}	
& {\nu^{\lambda+1}\sin\varphi}\, {| {\mathcal  R}_\nu^{(\lambda)} (\varphi)|}\le
\frac{\lambda(\lambda-1)2^{3/2\lambda}\Gamma(\lambda+1/2)}{\sqrt{\pi}}\frac{\nu^{\lambda+1}} {(\nu+1)^{\lambda+1}}\Big\{|\cos \varphi|+
\frac{2}{3}\frac{\lambda+1}{\nu+1}\sin \varphi\\
& \;\;\;\;  +\frac{ 2^{2-\lambda} \Gamma(2\lambda-1)}{\Gamma(\lambda+1)} \frac{(\nu+1)^{\lambda+1}}{(\nu-\lambda+3)^{2\lambda-1}}|\cot \varphi|^{\lambda-2}
\Big( |\cos \varphi|+\frac{2}{3}\frac{2\lambda-1}{\nu-\lambda+3}\sin \varphi\Big)\Big\}.
\end{split}\end{equation}
It is evident that
	\begin{equation}\begin{split}\label{AsympGeg-B9}	
|\cos \varphi|+
\frac{2}{3}\frac{\lambda+1}{\nu+1}\sin \varphi\le\frac{3\nu+2\lambda+5}{3(\nu+1)},\;\;
|\cos \varphi|+\frac{2}{3}\frac{2\lambda-1}{\nu-\lambda+3}\sin \varphi\le
\frac{3\nu+\lambda +7}{3(\nu-\lambda+3)}.
\end{split}\end{equation}
We write
	\begin{equation}\label{AsympGeg-B10}	
\frac{(\nu+1)^{\lambda+1}|\cot \varphi|^{\lambda-2}}{(\nu-\lambda+3)^{2\lambda-1}}=
\Big(\frac{\nu+1}{\nu-\lambda+3}\Big)^{\lambda+1}\,\Big(\frac{\nu}{\nu-\lambda+3}\Big)^{\lambda-2}
\,\Big(\frac{|\cot \varphi|}{\nu}\Big)^{\lambda-2}.
\end{equation}
 Using the inequality: $\ln(1+z)\le z$ for $z>-1$ again, we derive
 	\begin{equation}\begin{split}\label{AsympGeg-B11}	
\Big(\frac{\nu+1}{\nu-\lambda+3}\Big)^{\lambda+1}={\rm exp}\Big({(\lambda+1)\ln\Big(1+\frac{\lambda-2}{\nu-\lambda+3}\Big)}\Big)
\le {\rm exp}\Big(\frac{(\lambda-2)(\lambda+1)}{\nu-\lambda+3}\Big),
 \end{split}\end{equation}
 and
 	\begin{equation}\begin{split}\label{AsympGeg-B12}	
\Big(\frac{\nu}{\nu-\lambda+3}\Big)^{\lambda-2}={\rm exp}\Big({(\lambda-2)\ln\Big(1+\frac{\lambda-3}{\nu-\lambda+3}\Big)}\Big)
\le {\rm exp}\Big(\frac{(\lambda-3)(\lambda+1)}{\nu-\lambda+3}\Big).
 \end{split}\end{equation}
		By	\eqref{IntRepPartI-1},  we have
\begin{equation*}
\begin{split}
\frac{1}{\nu \sin\varphi}\le
\frac{\pi}{2}\max\Big\{\frac{1}{\nu \varphi},\frac{1}{\nu(\pi-\varphi)}\Big\}\le  \frac{c\pi}{2},
\end{split}\end{equation*}
which implies
\begin{equation}\label{AsympGeg-B13}
\begin{split}
\Big(\frac{|\cot \varphi|}{\nu}\Big)^{\lambda-2}=|\cos \varphi|^{\lambda-2}\Big(\frac{1}{\nu\sin\varphi }\Big)^{\lambda-2}
\le \Big(\frac{c\pi}{2}\Big)^{\lambda-2}.
\end{split}\end{equation}		
We therefore derive from the above  $B_\nu^{(\lambda)}$ in the second case.

\subsection{Proof of Theorem \ref{AsympGeg-CaseA}}\label{proof-Thm2.2}  
For $\theta\in (0,\pi),$ we can derive
the  bounds \eqref{AsympGeg-cA}-\eqref{AsympGeg-cB} from \eqref{uniformBndR} by multiplying $\sin\theta$ and $(\sin \theta)^{\lambda-1},$ respectively,  for two cases.

In order to derive the upper bounds uniform for both $\nu$ and $\theta,$ it is necessary  study the behaviors of ${}^{r\!}G_\nu^{(\lambda)}(x)$ at $x=\pm 1$ (i.e., $\theta=0,\pi$).
 It is evident that
by \eqref{rgjfdef}, ${}^{r\!}G_\nu^{(\lambda)}(1)=1$ for all $\lambda>-1/2$ and $\nu\ge0.$
We now  examine the behavior of right GGF-Fs at $x=-1.$ It is clear that
if $\nu=n\in {\mathbb N}_0,$ we have ${}^{r\!}G_n^{(\lambda)}(-1)=(-1)^n\, {}^{r\!}G_n^{(\lambda)}(1)=(-1)^n.$
We now consider the case with  $\nu \notin {\mathbb N}_0$.
Note that for  $-1/2<\lambda<1/2$  (cf. \cite[Prop. 2.3]{Liu2017arXiv}):
 \begin{equation}\label{defBS00}
		{}^{r\!}G_\nu^{(\lambda)}(-1) =
		\frac{ \cos((\nu+\lambda)\pi)}{\cos(\lambda\pi)},
		\end{equation}
so	${}^{r\!}G_\nu^{(\lambda)}(x)$ is continuous on $[-1,1].$ However, for $\lambda\ge 1/2$ and
$\nu\not\in {\mathbb N}_0$, ${}^{r\!}G_\nu^{(\lambda)}(x)$  is singular at $x=-1.$  Indeed, according to \cite[Prop. 2.3]{Liu2017arXiv}, we have
		\begin{equation}\label{defBS0A}
		\lim_{x\to -1^+}\frac{{}^{r\!}G_\nu^{(1/2)}(x)}{\ln(1+x)}= \frac{\sin (\nu\pi)} {\pi}\,,\quad \nu\not \in {\mathbb N}_0\,;  
		\end{equation}
	and for $\lambda>1/2$ and $\nu\not \in {\mathbb N}_0,$  we have
		\begin{equation}\label{defBS0B}
		\lim_{x\to -1^+} \Big(\frac{1+x} 2\Big)^{\lambda-1/2}\,  {}^{r\!}G_\nu^{(\lambda)}(x)=
		-\frac{\sin(\nu\pi)}\pi
		 \frac{\Gamma(\lambda-1/2)\Gamma(\lambda+1/2)\Gamma(\nu+1)}{\Gamma(\nu+2\lambda)}:={\mathcal Q}_\nu^{(\lambda)}.
		\end{equation}
Note that \eqref{defBS0B} also holds for $\nu=n\in {\mathbb N}_0,$ as ${\mathcal Q}_n^{(\lambda)}=0.$

We now consider the case with $\theta=0$.
  As ${}^{r\!}G_\nu^{(\lambda)}(1)=1$,  taking the limit $\theta\to 0,$ and
   find readily that the above bounds hold (note: ${\widetilde {\mathcal R}^{(\lambda)}_\nu}(0)=0$, but ${\widetilde {\mathcal S}^{(\lambda)}_\nu}(\theta)>0$ in \eqref{AsympGeg-cA}-\eqref{AsympGeg-cB}).


It remains to consider $\theta\to \pi^-,$ i.e., $x\to -1^+.$ Apparently, we have
$\sin\theta=\sqrt{1-x^2}.$ As a direct consequence of  \eqref{defBS00}-\eqref{defBS0A},
we have that for $0<\lambda\le 1/2,$
\begin{equation}\label{prfadd1}
\begin{split}
\lim_{\theta\to \pi^-} \widetilde{\mathcal R}_\nu^{(\lambda)}(\theta)& =
\lim_{\theta\to \pi^-}\big\{(\sin \theta)^{\lambda+1} \, {}^{r\!}G_\nu^{(\lambda)}(\cos\theta)\big\}
=\lim_{x\to -1^+} (1-x^2)^{(\lambda+1)/2} \, {}^{r\!}G_\nu^{(\lambda)}(x)=0.
\end{split}
\end{equation}
Similarly, by \eqref{defBS0B},  we have that for $1/2<\lambda<2,$
\begin{equation}\label{prfadd2}
\begin{split}
\lim_{\theta\to \pi^-} \widetilde{\mathcal R}_\nu^{(\lambda)}(\theta)&
=\lim_{x\to -1^+} \big\{(1-x^2)^{1-\lambda/2} (1-x^2)^{\lambda-1/2} \, {}^{r\!}G_\nu^{(\lambda)}(x)\big\}=0.
\end{split}
\end{equation}
For $\lambda\ge 2,$  we find from \eqref{defBS0B} that
\begin{equation}\label{prfadd4}
\begin{split}
\lim_{\theta\to \pi^-} \widetilde{\mathcal R}_\nu^{(\lambda)}(\theta)& =
\lim_{\theta\to \pi^-}\big\{(\sin \theta)^{2\lambda-1} \, {}^{r\!}G_\nu^{(\lambda)}(\cos\theta)\big\}
=\lim_{x\to -1^+} (1-x^2)^{\lambda-1/2} \, {}^{r\!}G_\nu^{(\lambda)}(x)\\
&= 2^{2\lambda-1}\,{\mathcal Q}_\nu^{(\lambda)}.
\end{split}
\end{equation}

(i) For $0<\lambda\le 2,\, \nu+\lambda>1$ and $\nu>0,$
we find from \eqref{BoundIntFunft-A} that
\begin{equation}\label{pfadd5}
\lim_{\theta\to \pi^-}{\widetilde {\mathcal S}^{(\lambda)}_\nu}(\theta)=\frac{\lambda|\lambda-1|2^\lambda\Gamma(\lambda+1/2)}{\sqrt{\pi}
(\nu+\lambda-1)^{\lambda+1}}.
\end{equation}
Thus, in this case, it is evident that by  \eqref{prfadd1}-\eqref{prfadd2} and \eqref{pfadd5},   \eqref{AsympGeg-cA} holds for $0<\lambda<2$.
For $\lambda=2,$  we obtain from \eqref{prfadd4}-\eqref{pfadd5} that
\begin{equation*}\label{prfadd5}
\begin{split}
\lim_{\theta\to \pi^-} \widetilde{\mathcal R}_\nu^{(2)}(\theta)
= 2^{3}\,{\mathcal Q}_\nu^{(2)}=-\frac{3\sin (\nu \pi)}{(\nu+1)(\nu+2)(\nu+3)},\quad
\lim_{\theta\to \pi^-}{\widetilde {\mathcal S}^{(2)}_\nu}(\theta)=\frac {6}{(\nu+1)^3}.
\end{split}
\end{equation*}
Hence, \eqref{AsympGeg-cA} holds for $\lambda=2.$

 (ii) For $\lambda>  2,\, \nu>\lambda-3$ and $\nu>0,$  we obtain from
\eqref{BoundIntFunft-B} that
\begin{equation}\label{AsympGeg-CaseA-2}
\lim_{\theta\to \pi^-}{\widetilde {\mathcal S}^{(\lambda)}_\nu}(\theta)=\frac{2^{2+\lambda/2}\Gamma(\lambda+1/2)\Gamma(2\lambda-1)}{\sqrt{\pi}(\nu-\lambda+3)^{2\lambda-1}\Gamma(\lambda-1)}
=\frac{\lambda\,2^{5\lambda/2}\Gamma(\lambda-1/2)\Gamma(\lambda+1/2)}{\pi(\nu-\lambda+3)^{2\lambda-1}},
\end{equation}
where we used the identity (cf. \cite[(6.1.18)]{Abramowitz1972Book}):
\begin{equation*}
	\Gamma(2z)=\pi^{-1/2}2^{2z-1}\Gamma(z)\Gamma(z+1/2).
	\end{equation*}
Using the inequality (cf. \cite[(5.6.7)]{Olver2010Book}): for $b - a \ge  1$, $a\ge 0$, and $z  > 0$,
$$ \frac{\Gamma(z+a)}{\Gamma(z+b)}\le z^{a-b}, $$
we get
\begin{equation}\label{AsympGeg-CaseA-3}
 \frac{\Gamma(\nu+1)}{\Gamma(\nu+2\lambda)}=\frac{\Gamma\big((\nu-\lambda+3)+(\lambda-2)\big)}
 {\Gamma\big((\nu-\lambda+3)+(3\lambda-3)\big)}\le(\nu-\lambda+3)^{1-2\lambda}\le \frac{\lambda\,2^{\lambda/2+1}}{(\nu-\lambda+3)^{2\lambda-1}}.
\end{equation}
 Thus, from \eqref{defBS0B} and \eqref{AsympGeg-CaseA-2}-\eqref{AsympGeg-CaseA-3}, we derive  that for $\lambda\ge 2,$
 $$\lim\limits_{\theta\to \pi^-} |\widetilde{\mathcal R}_\nu^{(\lambda)}(\theta)|= 2^{2\lambda-1}|{\mathcal Q}_\nu^{(\lambda)}|\le \lim_{\theta\to \pi^-}{\widetilde {\mathcal S}^{(\lambda)}_\nu}(\theta).$$
 This ends the proof.

\section{Some relevant  properties of GGF-Fs}\label{MiscProp}
\setcounter{equation}{0}
\setcounter{lmm}{0}
\setcounter{thm}{0}

 The GGF-Fs enjoy a rich collection of properties particularly in the fractional calculus  framework.
 In this section, we present assorted properties of GGF-Fs, and most of them follow directly from the properties of the hypergeometric functions.
These can  provide a better picture of this family of very useful special functions.
Recall  the definition  of  the  right-sided  Riemann-Liouville  fractional derivative of order $s>0$  (cf. \cite{Samko1993Book}):
\begin{equation}\label{leftintRL}
\begin{split}
{}_x^R D_{1}^s\, u(x)=(-1)^kD^k\big\{\,{}_xI_{1}^{k-s}\, u\big\}(x),\;\; s\in[k-1,k),
\end{split}
\end{equation}
where $D^k$ with $k\in {\mathbb N}$ is the ordinary $k$th derivative, and ${}_xI_1^\mu $ is
the RL fractional derivative operator defined in \eqref{IsIntegral}.
We have  the explicit formulas  $($cf.  \cite{Samko1993Book}$)$:   for  real $\eta>-1$ and $s>0,$
\begin{equation}\label{intformu}
{}_{x} I_{1}^s \, (1-x)^\eta=  \dfrac{\Gamma(\eta+1)}{\Gamma(\eta+s+1)} (1-x)^{\eta+s};\quad
{}_{x}^{R} D_{1}^s  \, (1-x)^\eta=\frac{\Gamma(\eta+1)}{\Gamma(\eta-s+1)} (1-x)^{\eta-s}.
\end{equation}

\begin{prop} \label{PropA}  {\bf (see \cite[Thm. 3.1]{Liu2017arXiv}).}
For  real $\lambda>s-1/2,$ real $\nu\ge 0$ and $x\in(-1,1),$
		\begin{equation}\label{dFCI++}	
		{}_{x}^R D_{1}^{s}\big\{(1-x^2)^{\lambda-1/2}
		\,{}^{r\!}G_{\nu}^{(\lambda)}(x)\big\}
		=\frac{2^s \,\Gamma(\lambda+1/2)}{\Gamma(\lambda-s+1/2)}\,(1-x^2)^{\lambda-s-1/2}\, {}^{r\!}G_{\nu+s}^{(\lambda-s)}(x).
			\end{equation}
\end{prop}

Note that we just list the properties for the right GGF-F ${}^{r\!}G_{\nu}^{(\lambda)}(x),$
but similar formulas are valid for the left GGF-F ${}^{l}G_{\nu}^{(\lambda)}(x)$ (cf.
\eqref{lgjfdef}) under the left RL fractional derivative (cf. \cite{Liu2017arXiv}).

As a generalization of Gegenbauer polynomials, the GGF-Fs satisfy the following fractional  Rodrigues' formula.
\begin{prop} \label{PropRGGF}  	 For real $\lambda>-1/2 $ and real   $\nu\ge 0,$  the GGF-Fs defined in  \eqref{rgjfdef} satisfy
\begin{equation}\begin{split}
	\label{PropRGGF-0} %
	{}^{r\!}G_\nu^{(\lambda)}(x) & =
	\frac{\Gamma(\lambda+1/2)} {2^\nu \,\Gamma(\nu+\lambda+1/2)} (1-x^2)^{-\lambda+1/2} \,
 {}_{x}^{R} D_{1}^\nu  \big\{  (1-x^2)^{\nu+\lambda-1/2}  \big\}.
	\end{split}\end{equation}
\end{prop}
\begin{proof}	
	Substituting $\nu, \lambda, s$ in \eqref{dFCI++} by $0, \nu+\lambda, \nu,$ respectively, yields
		\begin{equation*}
	{}_{x}^R D_{1}^{\nu}\big\{(1-x^2)^{\nu+\lambda-1/2}
	\big\}
	=\frac{2^\nu \,\Gamma(\nu+\lambda+1/2)}{\Gamma(\lambda+1/2)}\,(1-x^2)^{\lambda-1/2}\, {}^{r\!}G_{\nu}^{(\lambda)}(x),\\
	\end{equation*}
	which implies   \eqref{PropRGGF-0}.    
\end{proof}
\begin{rem}\label{correctA} {\em Mirevski et al \cite[Definition 9]{Mirevski2007AMC} defined  the {\rm(}generalized or{\rm)} $g$-Jacobi function through the (fractional) Rodrigues' formula
		and derived  an equivalent representation in terms of the hypergeometric function {\rm(}cf.
\cite[Thm. 12]{Mirevski2007AMC}{\rm).} However, we point out that
the left RL fractional derivative operator ${}_{0}^{R} D_{x}^\nu$ therein should be  replaced by the right RL fractional derivative operator
${}_{x}^{R} D_{1}^\nu$ as in \eqref{PropRGGF-0}. Then the flaws in the derivation of \cite[Thm. 12]{Mirevski2007AMC}
can be fixed accordingly. }
\end{rem}

\begin{prop} \label{PropRGGFB}  	 For real $\lambda>-1/2 $ and real   $\nu\ge 0,$
the GGF-Fs have the series representation:
\begin{equation}\begin{split}
	\label{PropRGGF-1} %
	{}^{r\!}G_\nu^{(\lambda)}(x)
 &=\frac{\Gamma(\lambda+1/2)\Gamma(\nu+1)}{2^{\nu}\Gamma(\nu+\lambda+1/2)} \sum_{k=0}^{\infty}	 \binom{\nu+\lambda-1/2}{\nu-k} \binom{\nu+\lambda-1/2}{k} (x-1)^{k}(1+x)^{\nu-k}.
	\end{split}\end{equation}
\end{prop}

\begin{proof}
	Using the fractional Leibniz rule (cf. \cite[(2.202)]{Podlubny1999Book}), we obtain from
\eqref{intformu} that
	\begin{align*}
	\label{RJacobiR1}
	\begin{split}
{}_{x}^{R} D_{1}^\nu  \big\{  (1-x^2)^{\nu+\lambda-1/2}  \big\}&	=
	\sum_{k=0}^{\infty} \binom{\nu}{k}   \,{}_{x}^{R} D_{1}^{\nu-k} (1-x)^{\nu+\lambda-1/2}\,
	(-1)^kD^k  (1+x)^{\nu+\lambda-1/2}
	\\
	& =(1-x^2)^{\lambda-1/2}\sum_{k=0}^{\infty} \binom{\nu}{k}
	\frac{\Gamma^2(\nu+\lambda+1/2) (x-1)^{k}(1+x)^{\nu-k}}{\Gamma(k+\lambda+1/2)\Gamma(\nu-k+\lambda+1/2)}.
	\end{split}
	\end{align*}
	Recall the definition of the binomial coefficient
		\begin{equation*}\label{fomula3}
		\binom{\nu}{k}=\frac{\Gamma(\nu+1)}{\Gamma(\nu-k+1)\Gamma(k+1)}.
		\end{equation*}
Thus,	we have
		\begin{equation*}\label{fomula4}
\binom{\nu+\lambda-1/2}{\nu-k} =
	\frac{\Gamma(\nu+\lambda+1/2)}{\Gamma(k+\lambda+1/2)\Gamma(\nu-k+1)},
\;\; \binom{\nu+\lambda-1/2}{k}=    \frac{\Gamma(\nu+\lambda+1/2) }{\Gamma(\nu-k+\lambda+1/2)\Gamma(k+1)}.
			\end{equation*}
Then  \eqref{PropRGGF-1} follows from the above.
\end{proof}
\begin{rem}\label{alternativeS}{\em Alternatively, we can derive  \eqref{PropRGGF-1}
from \eqref{hyperboscs}, Definition \ref{defnGGG}, and the Pfaff's formula {\rm(}cf. \cite[Theorem 2.2.5]{Andrews1999Book}{\rm):}
$${}_2F_1(a,b;c; z)=(1-z)^{-a}{}_2F_1(a,c-b; c; z/(1-z)). $$}
\end{rem}

We next  present some recurrence relations that generalize the corresponding formulas for
the Gegenbauer polynomials.
\begin{prop}\label{ThreeTermRelation} For
real $\lambda> -1/2,$
	the  GGF-Fs satisfy the recurrence formulas
\begin{equation}	\begin{split}\label{ThreeTermRelation-0}
(\nu+2\lambda)\,{}^{r\!}G_{\nu+1}^{(\lambda)}(x)=2(\nu+\lambda)\,x\,{}^{r\!}G_\nu^{(\lambda)}(x)-\nu \,{}^{r\!}G_{\nu-1}^{(\lambda)}(x),\;\; \nu \ge 1,
\end{split}\end{equation}
and
	\begin{equation}	\begin{split}\label{ThreeTermRelationB-0}
&{}^{r\!}G_{\nu}^{(\lambda)}(x)=x\,{}^{r\!}G_{\nu-1}^{(\lambda+1)}(x)-\frac{(\nu-1)(\nu+2\lambda+1)}{4(\lambda+1/2)(\lambda+3/2)}(1-x^2)\,{}^{r\!}G_{\nu-2}^{(\lambda+2)}(x),
\;\;\nu\ge 2.
\end{split}\end{equation}
\end{prop}
\begin{proof}
Recall the formula (cf. \cite[(2.5.15)]{Andrews1999Book}):
\begin{equation}	\begin{split}\label{ThreeTermRelation-1}
&2b(c-a)(b-a-1)\,{}_2F_1(a-1,b+1;c;z)\\
&\quad -\big((1-2z)(b-a-1)_3+(b-a)(b+a-1)(2c-b-a-1)\big)\, {}_2F_1(a,b;c;z)\\
&\quad\quad \quad-2a(b-c)(b-a+1)\,{}_2F_1(a+1,b-1;c;z)=0.
\end{split}\end{equation}
Substituting $a,$ $b,$ $c$ and $z$ in \eqref{ThreeTermRelation-1} by $-\nu,$ $ \nu+2\lambda,$ $\lambda+1/2$ and $(1-x)/2$, respectively, and using the definition \eqref{rgjfdef}, we obtain %
\begin{equation}	\begin{split}\label{ThreeTermRelation-2}
&2(\nu+2\lambda)(\nu+\lambda+1/2)(2\nu+2\lambda-1)\,{}^{r\!}G_{\nu+1}^{(\lambda)}(x)-(2\nu+2\lambda-1)_3\,x\,{}^{r\!}G_\nu^{(\lambda)}(x)\\
&\quad\quad +2\nu (\nu+\lambda+1/2)(2\nu+2\lambda-1)\,{}^{r\!}G_{\nu-1}^{(\lambda)}(x)=0,
\end{split}\end{equation}
which implies  \eqref{ThreeTermRelation-0}.

	Recall  (cf. \cite[(2.5.2)]{Andrews1999Book})
	\begin{equation}	\begin{split}\label{ThreeTermRelationB-1}
	&z(1-z)\frac{(a+1)(b+1)}{c(c+1)}{}_2F_1(a+2,b+2;c+2;z)\\
	&\quad +\frac{(c-(a+b+1)z)}{c}{}_2F_1(a+1,b+1;c+1;z)-{}_2F_1(a,b;c;z)=0.
	\end{split}\end{equation}
	Substituting $a,$ $b,$ $c$ and $z$ in \eqref{ThreeTermRelationB-1} by $-\nu,$ $ \nu+2\lambda,$ $\lambda+1/2$ and $(1-x)/2$, respectively, leads to
	\begin{equation*}	\begin{split}
&{}^{r\!}G_{\nu}^{(\lambda)}(x)=x\,{}^{r\!}G_{\nu-1}^{(\lambda+1)}(x)-\frac{(\nu-1)(\nu+2\lambda+1)}{4(\lambda+1/2)(\lambda+3/2)}(1-x^2)\,{}^{r\!}G_{\nu-2}^{(\lambda+2)}(x).
\end{split}\end{equation*}
This completes the proof.
\end{proof}

\begin{prop}\label{SturmLiu} For real $\lambda>-1/2$ and real $\nu\ge 0,$ the  GGF-Fs
	satisfy the Sturm-Liouville equation
\begin{equation}\label{2ndDE-B}
(1-x^2)(	{}^{r\!}G_\nu^{(\lambda)}(x))''-(2\lambda+1) x\, (	 {}^{r\!}G_\nu^{(\lambda)}(x))'+\nu(\nu+2\lambda)\,	{}^{r\!}G_\nu^{(\lambda)}(x)=0,
 \end{equation}
or equivalently,
\begin{equation}\label{2ndDE2-B}
\big\{(1-x^2)^{\lambda+1/2}\,(	 {}^{r\!}G_\nu^{(\lambda)}(x))'\big\}'+\nu(\nu+2\lambda)(1-x^2)^{\lambda-1/2}\,  	 {}^{r\!}G_\nu^{(\lambda)}(x)=0.
\end{equation}
\end{prop}
\begin{proof}
Note that $F:={}_2F_1(a,b;c; z)$ satisfies the second-order  equation (cf. \cite[P. 94]{Andrews1999Book}):
\begin{equation}\label{SecondODE-B}
z(1-z)F''+\big(c-(a+b+1)z\big)F'-abF=0.
\end{equation}
Substituting $a,$ $b,$ $c$ and $z$ in \eqref{SecondODE-B} by $-\nu,$ $ \nu+2\lambda,$ $\lambda+1/2$ and $(1-x)/2$, respectively, we derive \eqref{2ndDE-B} from
the definition
 \eqref{rgjfdef}.
\end{proof}

Similar to the Gegenbauer polynomials, we have the following derivative relations.
\begin{prop} For real $\nu\ge k\in {\mathbb N},$ we have
\begin{equation}\label{dProperties-2}
\begin{split}
\frac{d^k}{dx^k}{}^{r\!}G_\nu^{(\lambda)}(x)
=(-1)^k\frac{(-\nu)_k(\nu+2\lambda)_k}{2^k(\lambda+1/2)_k} \, {}^{r\!}G_{\nu-k}^{(\lambda+k)}(x).
\end{split}\end{equation}
In particular, if $k=1,$  we have 
	\begin{equation}\label{dProperties-3}
\begin{split}
\frac{d}{dx}{}^{r\!}G_\nu^{(\lambda)}(x)
=	\frac{\nu(\nu+2\lambda)}{2\lambda+1}{}^{r\!}G_{\nu-1}^{(\lambda+1)}(x),\;\;\nu\ge 1.
\end{split}\end{equation}
\end{prop}
\begin{proof} The formula \eqref{dProperties-2} is derived directly from  the identity (cf. \cite[(15.5.2)]{Olver2010Book}):
\begin{equation}\label{dProperties-1}\frac{d^k}{dz^k} {}_2F_1(a, b;c;z)=\frac{(a)_k(b)_k}{(c)_k} {}_2F_1(a+k, b+k;c+k;z),
\end{equation}
and the definition \eqref{rgjfdef}.
\end{proof}



For completeness, we quote  the following estimates,  which were very useful in the error analysis in \cite{Liu2017arXiv}.
\begin{prop}\label{LBoundForGegPoly}  {\bf (see \cite[Thms 2.1-2.2]{Liu2017arXiv}).}
	For   $0< \lambda < 1$ and real $\nu\ge 0$, we have
\begin{equation}\label{BoundGeg-C}
\max_{|x|\le 1}\big\{	(1-x^2)^{\lambda/2} \big|{}^{r\!}G_\nu^{(\lambda)}(x)\big|\big\}\le
	\varrho_\nu^{(\lambda)},
\end{equation}
where
\begin{equation}\label{kappa-C}
\begin{split}
\varrho_\nu^{(\lambda)}=\frac{\Gamma(\lambda+1/2)}{\sqrt \pi}\bigg(
\frac{ \cos^2(\pi\nu/ 2)\Gamma^2(\nu/2+1/2) }{\Gamma^2( ({\nu}+1)/2+\lambda)}+\frac{4\sin ^2\big(\pi  {\nu} /2\big)}{\nu^2+2\lambda \nu+\lambda}\frac{\Gamma^2(\nu/2+1)}{\Gamma^2({\nu}/2+\lambda)}\bigg)^{1/2}.
\end{split}	\end{equation}
For $\lambda \ge 1$ and real $\nu\ge 0$, we have
\begin{equation}\label{maxvalueS}
\max_{|x|\le 1}\big\{(1-x^2)^{\lambda-1/2}\big|{}^{r\!}G_\nu^{(\lambda)}(x)\big| \big\}\le
\kappa_\nu^{(\lambda)}
,
\end{equation}
where
\begin{equation}\label{kappaNl}
\begin{split}
\kappa_\nu^{(\lambda)}=\frac{\Gamma(\lambda+ 1/2)}{\sqrt \pi}\bigg(
\frac{ \cos^2(\pi\nu/ 2) \Gamma^2( ({\nu}+1)/ 2)}{\Gamma^2( ({\nu}+1)/2+\lambda)}+\frac{4\sin ^2\big(\pi  {\nu} /2\big)}{2\lambda-1+\nu(\nu+2\lambda)}\frac{\Gamma^2({\nu}/2+1)}{\Gamma^2({\nu}/2+\lambda)}\bigg)^{1/2}.
\end{split}	\end{equation}
\end{prop}

\medskip

\begin{appendix}
\section{Proof of Lemma  \ref{UperBoundf} }\label{AppendixA}
\renewcommand{\theequation}{A.\arabic{equation}}

We first show that
\begin{equation}\begin{split}\label{IntRep-13}			
\frac{t^2}{4}\cos^2\varphi\, (\cosh(t/2))^{\frac 43} +\sin^2 \varphi(\cosh  t)^{\frac 23}<|g(\varphi, t)|^2
<\frac{t^2}{4}\cos^2\varphi\cosh^4 (t/2)
+\sin^2 \varphi\cosh^2 t,
\end{split}\end{equation}
and
\begin{equation}\begin{split}\label{IntRep-15}	
|\partial_t g(\varphi,t)|\le& \Big(\frac{t}{3} \sin\varphi + \frac{1}{2} |\cos\varphi| \Big) \cosh t.
\end{split}\end{equation}

It is clear that
\begin{equation}\label{IntRep-9}\begin{split}
|g(\varphi, t)|^2=\frac{\cos^2\varphi (\cosh t-1)^2+ \sin^2 \varphi\sinh^2 t}{t^2} .
\end{split}\end{equation}	
Recall the properties of hyperbolic functions (cf. \cite[(4.32.1), (4.32.2), (4.35.20)]{Olver2010Book}): for $t>0,$
\begin{equation}\label{IntRep-10}\begin{split}
(\cosh t)^{\frac 13}<\frac{\sinh t}{t};  \quad \tanh t<t;\quad  \sinh \frac{t}{2}=\Big(\frac{\cosh t-1}{2}\Big)^{\frac 1 2}.
\end{split}\end{equation}	
Then we derive
\begin{equation}\label{IntRep-11}\begin{split}
(\cosh t)^{\frac 13}<\frac{\sinh t}{t}<\cosh t,\quad \forall\, t>0,
\end{split}\end{equation}	
and
\begin{equation}\begin{split}	\label{IntRep-12}		
\frac{1}{2}\big(\cosh (t/2)\big)^{\frac 2 3}<\frac{\cosh t-1}{t^2} =\frac{1}{2}\Big(\frac{\sinh( t/2)}{t/2}\Big)^2<\frac{1}{2} \cosh^2(t/2).
\end{split}\end{equation}
Thus we obtain \eqref{IntRep-13} from \eqref{IntRep-9} and \eqref{IntRep-11}-\eqref{IntRep-12} immediately.

A direct calculation from \eqref{gttf} yields
\begin{equation}\label{derivative-g}		
\begin{split}	
\partial_t g(\varphi, t) 
= \frac{\cos \varphi\, (t\sinh t-\cosh t+1) +\ri\sin\varphi\, (t\cosh t-\sinh t)}{t^2},
\end{split}\end{equation}
and
\begin{equation}\begin{split}	\label{absderivative-g}		
|\partial_t g(\varphi, t)|^2=\frac{\cos^2 \varphi\, (t\sinh t-\cosh t+1)^2+\sin^2\varphi\, (t\cosh t-\sinh t)^2}{t^4} .
\end{split}\end{equation}
We next  show that for $t>0$,		
\begin{equation}\begin{split}\label{Bound-A}			
\frac{ t\cosh t-\sinh t}{t^3}<\frac{1}{3}\cosh t,
\end{split}\end{equation}
and		
\begin{equation}\begin{split}\label{Bound-B}		
\frac{t\sinh t-\cosh t+1}{t^2}<\frac{1}{2}\cosh t.
\end{split}\end{equation}
To prove \eqref{Bound-A},  we denote
$h(t):=t^3\cosh t- 3t\cosh t+3\sinh t.$  Then  for $t>0,$
\begin{equation}\begin{split}\label{dht}	
h'(t)= t^3\sinh t+3t(t\cosh t - \sinh t)>t^3\sinh t>0,
\end{split}\end{equation}	
where we used the property:  $t\cosh t>\sinh t$ (cf. \eqref{IntRep-10}). Therefore, $h(t)$ is strictly ascending, so for all  $t>0,$
$$ h(t)=t^3\cosh t- 3t\cosh t+3\sinh t>h(0)=0,$$
which implies \eqref{Bound-A}.  As  $$(t\sinh t-\cosh t+1)'=t\cosh t >0,\;\;\; t>0,$$
we have $t\sinh t-\cosh t+1 >0$ for all $t>0.$ Denoting
$\hat h(t):=t^2\cosh t- 2t\sinh t+2\cosh t-2,$  we find   for $t>0,$
\begin{equation}\begin{split}\label{dhatht}	
\hat h'(t)=t^2\sinh t>0,\;\; {\rm so} \;\;  \hat h(t)> \hat h(0)=0,
\end{split}\end{equation}	
which yields  \eqref{Bound-B}.

From \eqref{absderivative-g}, \eqref{Bound-A} and \eqref{Bound-B}, we obtain
\begin{equation}\begin{split}\label{IntRep-14}	
|\partial_t g(\varphi, t)|^2\le&\frac{1}{9} t^2\sin^2\varphi \cosh^2 t+ \frac{1}{4} \cos^2\varphi\cosh^2 t,
\end{split}\end{equation}
which leads to  \eqref{IntRep-15}.

Now, we are ready to derive \eqref{ftbnds}-\eqref{ftbndsB}.
Using the mean-value theorem for the real part and imaginary part of $f^{(\lambda)}(\varphi,t)$, respectively, we obtain
\begin{equation}\label{UperBoundf-1}
\begin{split}
f^{(\lambda)}(\varphi,t)=\frac{g^{\lambda-1}(\varphi, t)-	g^{\lambda-1}(\varphi, 0)} t
=\RE \big\{\partial_t	g^{\lambda-1}(\varphi, \xi_1)\big\}+\ri \,\IM \big\{\partial_t	 g^{\lambda-1}(\varphi, \xi_2)\big\},
\end{split}\end{equation}
for $\xi_i=\xi_i(t)\in (0,t), i=1,2,$ and $\varphi\in (0,\pi).$ Hence, we have
\begin{equation}\label{UperBoundf-2}
\begin{split}
|f^{(\lambda)}(\varphi,t)|\le 2\sup_{0<\xi<t}|\partial_t	g^{\lambda-1}(\varphi, \xi)|=2|\lambda-1|\sup_{0<\xi< t}\big\{|g(\varphi,\xi)|^{\lambda-2}|\partial_t g(\varphi,\xi)|\big\}.
\end{split}\end{equation}
We now estimate its upper bound.  From \eqref{IntRep-15}, we obtain that for $\xi\in (0,t)$ and $\varphi\in (0,\pi),$
\begin{equation}\begin{split}\label{UperBoundf-3}
|\partial_t g(\varphi,\xi)|\le \Big(\frac{\xi}{3} \sin\varphi + \frac{1}{2} |\cos\varphi| \Big) \cosh \xi
\le \Big(\frac{t}{3} \sin\varphi + \frac{1}{2} |\cos\varphi| \Big) e^ t.
\end{split}\end{equation}
It remains to estimate the upper bound of $|g(\varphi, \xi)|^{\lambda-2}.$  We proceed with two cases.

\vskip 4pt
i) For $0<\lambda\le 2,$   we obtain from the lower bound $g$ in   \eqref{IntRep-13} that  for  $0<\xi<t$,
\begin{equation}\begin{split}\label{UperBoundf-4}
|g(\varphi, \xi)|^{\lambda-2}\le \Big(\frac{1}{4}\cos^2\varphi \cosh ^{4/3} (\xi/2)\xi^2 +\sin^2 \varphi\cosh^{2/3}  \xi\Big)^{\lambda/2-1}
\le\sin^ {\lambda-2}\varphi,
\end{split}\end{equation}
where we used the fact  the function in $\xi$ is strictly decreasing,  since $\lambda/2-1<0.$ Thus, we obtain \eqref{ftbnds} from
\eqref{UperBoundf-2}-\eqref{UperBoundf-4}.

\vskip 4pt

ii) For $\lambda> 2$, we obtain from the upper bound of $g$ in   \eqref{IntRep-13} that
\begin{equation}\begin{split}\label{UperBoundf-5}	
|g(\varphi, \xi)|^{\lambda-2}& \le\Big(\frac{1}{4}\cos^2\varphi\cosh^4 (\xi/2) \xi^2
+\sin^2 \varphi\cosh^2 \xi\Big)^{\lambda/2-1}\\
& \le\Big(\frac{1}{4}\cos^2\varphi\cosh^4 (t/2) t^2
+\sin^2 \varphi\cosh^2 t\Big)^{\lambda/2-1}\\
& \le \Big(\max\Big \{\frac{1}{2}\cos^2\varphi\cosh^4 (t/2) t^2,2\sin^2 \varphi\cosh^2 t\Big\}\Big)^{\lambda/2-1}\\
& \le \Big(\max\Big \{\frac{1}{2}t^2\cos^2 \varphi \cosh^4 t,2\sin^2 \varphi\cosh^2 t\Big\}\Big)^{\lambda/2-1}\\
& \le 2^{1-\lambda/2}(\cos\varphi)^{\lambda-2}(\cosh t) ^{2\lambda-4}t^{\lambda-2}+2^{\lambda/2-1}(\sin \varphi)^{\lambda-2}(\cosh t)^{\lambda-2}\\
&\le  2^{1-\lambda/2}(\cos\varphi)^{\lambda-2} e^{2(\lambda-2)t} t^{\lambda-2}+2^{\lambda/2-1}(\sin \varphi)^{\lambda-2} e^{(\lambda-2)t}\\
& = 2^{\lambda/2-1}(\sin \varphi)^{\lambda-2} e^{(\lambda-2)t}\Big(1+\frac{|\cot \varphi|^{\lambda-2}} {2^{\lambda-2}}\, t^{\lambda-2} e^{(\lambda-2)t}\Big).
\end{split}\end{equation}
Therefore, we obtain \eqref{ftbndsB} from   \eqref{UperBoundf-2}-\eqref{UperBoundf-3}  and \eqref{UperBoundf-5}.

\end{appendix}


 \end{document}